\documentclass[reqno]{amsart}
\usepackage{amssymb}
\usepackage[usenames, dvipsnames]{color}
\usepackage{pxfonts}
\usepackage{enumerate}
\usepackage{amsmath}
\usepackage[driverfallback=dvipdfm]{hyperref}{\tiny}
\usepackage{verbatim}

\allowdisplaybreaks[4]

\numberwithin{equation}{section}

\newtheorem{theorem}{Theorem}[section]
\newtheorem{corollary}[theorem]{Corollary}
\newtheorem{lemma}[theorem]{Lemma}
\newtheorem{prop}[theorem]{Proposition}

\theoremstyle{definition}
\newtheorem{remark}[theorem]{Remark}

\theoremstyle{definition}
\newtheorem{definition}[theorem]{Definition}

\theoremstyle{definition}
\newtheorem{assumption}[theorem]{Assumption}

\theoremstyle{definition}

\makeatletter
\def\dashint{\operatorname%
{\,\,\text{\bf-}\kern-.98em\DOTSI\intop\ilimits@\!\!}}
\makeatother

\def\\det{\text{det}}

\def\.5{\frac{1}{2}}

\newcommand{\RN}[1]{%
  \textup{\uppercase\expandafter{\romannumeral#1}}%
}

\newcommand{\Div}{\operatorname{div}}

\newcounter{marnote}

\makeatletter
\@namedef{subjclassname@2020}{%
  \textup{2020} Mathematics Subject Classification}
\makeatother

\begin{document}


\title[Gradient estimates]{Gradient continuity estimates for elliptic equations of singular $p$-Laplace type with measure data}

\author[L. Xu]{Longjuan Xu}
\address[L. Xu]{Academy for Multidisciplinary Studies, Capital Normal University, Beijing 100048, China.}
\email{longjuanxu@cnu.edu.cn}
\thanks{L. Xu was partially supported by NSF of China (12301141), and Beijing Municipal Education Commission Science and Technology Project (KM202410028001).}

\author[Y. Zhao]{Yirui Zhao}
\address[Y. Zhao]{School of Mathematical Sciences, Capital Normal University, Beijing 100048, China.}
\email{2230502155@cnu.edu.cn}

\begin{abstract}
In this paper, we are concerned with elliptic equations of
$p$-Laplace type with measure data, which is given by $-\Div\big(a(x)(|\nabla u|^2+s^2)^{\frac{p-2}{2}}\nabla u\big)=\mu$ with $p>1$ and $s\geq0$. Under the assumption that the modulus of continuity of the coefficient $a(x)$ in the $L^2$-mean sense satisfies the Dini condition, we prove a new comparison estimate and use it to derive interior and global gradient pointwise estimates by Wolff potential for $p\geq 2$ and Riesz potential for $1<p<2$, respectively. Our interior gradient pointwise estimates can be applied to a class of singular quasilinear elliptic equations with measure data given by $-\Div(A(x,\nabla u))=\mu$. We generalize the results in the papers of Duzaar and Mingione [Amer. J. Math. 133, 1093--1149 (2011)], Dong and Zhu [J. Eur. Math. Soc. 26, 3939--3985 (2024)], and Nguyen and Phuc [Arch. Rational Mech. Anal. (2023) 247:49], etc., where the coefficient is assumed to be Dini continuous. Moreover, we establish interior and global modulus of continuity estimates of the gradients of solutions.  
\end{abstract}

\maketitle

\section{Introduction and main results}
In this paper, we consider elliptic equations of
$p$-Laplace type with measure data and $s\geq0$:
\begin{equation}\label{p-laplace}
-\Div\big(a(x)(|\nabla u|^2+s^2)^{\frac{p-2}{2}}\nabla u\big)=\mu\quad\mbox{in}~\Omega,
\end{equation}
where $\Omega$ is a domain in $\mathbb R^n$ with $n\geq 2$, $p>1$, $\mu$ is a locally finite signed Radon measure in $\Omega$, that is, $|\mu|(B_R(x)\cap\Omega)<\infty$ for any $B_R(x)\subset\mathbb R^n$. We shall always assume that $\mu$ is defined in $\mathbb R^n$ by setting $|\mu|(\mathbb R^n\setminus\Omega)=0$. The parameter $s\geq0$ is fixed distinguishing the degenerate case ($s=0$) from the non-degenerate one ($s>0$). We say $u\in W_{\text{loc}}^{1,p}(\Omega)$ is a weak solution of \eqref{p-laplace} if for any $\varphi\in C_0^{\infty}(\Omega)$,
\begin{equation*}
\int_{\Omega}\langle a(x)(|\nabla u|^2+s^2)^{\frac{p-2}{2}}\nabla u),\nabla \varphi\rangle dx=\int_{\Omega}\varphi\ d\mu
\end{equation*}
holds. The focus of this paper is on deriving pointwise estimates for the gradient of the weak solution to \eqref{p-laplace} via Wolff potential 
\begin{equation}\label{W potential}
{\bf W}_{\beta,p}^R(|\mu|)(x)=\int_0^R\big(\frac{|\mu|(B_t(x))}{t^{n-\beta p}}\big)^{\frac{1}{p-1}}\frac{dt}{t},\quad\beta\in(0,n/p],~p\geq 2,
\end{equation} 
and Riesz potential
\begin{equation}\label{I potential}
{\bf I} _1^R(|\mu|)(x)=\int_0^R\frac{|\mu|(B_t(x))}{t^{n-1}}\frac{dt}{t},\quad1<p<2.
\end{equation}
The estimates provided by these potentials are non-linear analog of the classical pointwise estimate of the gradient of solutions to the Poisson equation $-\Delta u=\mu$, where $\nabla u$ can be estimated by Riesz potential \cite{km2012,km2014}. The second aim of this paper is to demonstrate $C^1$-regularity of solutions giving conditions in terms of Wolff potential and Riesz potential implying the continuity of $\nabla u$. Furthermore, we will establish interior and global modulus of continuity estimates of the gradients of solutions. 
Although we focus on the model case in \eqref{p-laplace}, our method can be applied to equations with more general structures: 
\begin{equation}\label{general-eq}
-\Div(A(x,\nabla u))=\mu.
\end{equation}
The details of this generalization are provided in Remark \ref{rmk-general}.  

The case where $p=2$ has been studied in \cite{m2011}. Subsequently, in \cite{dm2011,MR2010}, the study was extended to the case where $p\geq2$, with \eqref{general-eq} under consideration. To ensure the comprehensiveness of this paper, we will review the results for the model case presented in \eqref{p-laplace}. In the groundbreaking work \cite{dm2011}, Duzaar and Mingione established the gradient pointwise estimate for solutions to \eqref{p-laplace} as follows: for almost $x\in\Omega$,
\begin{equation*}
|\nabla u(x)|\leq C{\bf W}_{\frac{1}{p},p}^R(|\mu|)(x)+C\fint_{B_R(x)}(|\nabla u(y)|+s)\ dy
\end{equation*}
holds whenever $B_R(x)\subset\Omega$, given that $p\geq2$ and $a(x)$ exhibits Dini continuity. In a follow-up study \cite{MR2010}, they showed that $\nabla u$ is continuous in an open subset $\Omega_0$ of $\Omega$, by assuming that the functions
\begin{equation*}
    x\rightarrow{\bf W}_{\frac{1}{p},p}^R(|\mu|)(x)~converge~locally~uniformly~to~zero~in~\Omega_0~as~R\rightarrow 0.
\end{equation*}
For further related results concerning $p\geq2$, we refer to \cite{DM2009,m2011-2,Kuusi2013} and the references therein. 

Duzaar and Mingione \cite{dm2010} initially considered the case where $p\in\big(2-\frac{1}{n},2\big)$ and proved that under the assumption of Dini continuity for the coefficient and $u\in C^1$, the following pointwise gradient estimate holds for almost $x\in\Omega$:
\begin{equation*}
    |\nabla u(x)|\leq C\fint_{B_R(x)}(|\nabla u|+s)\ dy+C[{\bf I}_1^R(|\mu|)(x)]^{\frac{1}{p-1}},\quad B_R(x)\subset\Omega.
\end{equation*}
Nguyen and Phuc \cite{MR2020} extended the study to the case $p\in\big(\frac{3n-2}{2n-1},2-\frac{1}{n}\big]$ and obtained a pointwise gradient estimate of $C^1$-solution in terms of a truncated nonlinear Wolff potential as follows:
\begin{equation*}
|\nabla u(x)|\leq C\left(\fint_{B_R(x)}(|\nabla u|+s)^\gamma\ dy\right)^{1/\gamma}+C[{\bf P}_\gamma^R(|\mu|)(x)]^{\frac{1}{\gamma(p-1)}},\quad B_R(x)\subset\Omega,
\end{equation*}
where $\gamma\in\big(\frac{n}{2n-1},\frac{n(p-1)}{n-1}\big)$ and  
$${\bf P}_\gamma^R(|\mu|)(x):=\int_0^R\big(\frac{|\mu|(B_t(x))}{t^{n-1}}\big)^{\gamma}\frac{dt}{t}.$$
Recently, Dong and Zhu \cite{dz2024} further extended the results in \cite{MR2020} to the case $p\in\big(\frac{3n-2}{2n-1},2\big)$ with the estimate
\begin{equation*}
|\nabla u(x)|\leq C\left(\fint_{B_R(x)}(|\nabla u|+s)^{2-p}\ dy\right)^{\frac{1}{2-p}}+C[{\bf I}_1^R(|\mu|)(x)]^{\frac{1}{p-1}},\quad B_R(x)\subset\Omega,
\end{equation*}
and they also derived the Lipschitz estimates for the case $p\in(1,2)$. Subsequently,  Nguyen and Phuc \cite{np2023} established a  comparison estimate in the range $p\in(1,3/2)$ and then proved the pointwise bounds of solutions for the case $p\in(1,3/2)$. 

When $p=2$, the equation \eqref{p-laplace} reduces to the linear case : 
\begin{equation*}
Lu:=-\Div(a(x)\nabla u).
\end{equation*}
It is a well-established fact that weak solutions to  $Lu=\Div g$ are continuously differentiable \cite{b1978,mm2011} under the assumption that both the coefficient and the data $g$ are $\alpha$-increasing Dini continuous for some $\alpha\in(0,1]$. Dong and Kim \cite{dk2017} extended this result by demonstrating that weak solutions to $Lu=\Div g$ are continuously differentiable provided that the modulus of continuity of the coefficient and $g$ in the $L^1$-mean sense satisfies the Dini condition, which answers a question raised by Li \cite{li2017}. The results in \cite{dk2017} prompt the question of whether the assumption of Dini-continuity for $a(\cdot)$ in \eqref{p-laplace} is sharp? 
In this paper, we aim to show that the conditions on the coefficient can indeed be relaxed. The main findings of the paper, which include both interior and boundary gradient estimates, are presented in Subsections \ref{interior-result} and \ref{boundary-result}, respectively. 

\subsection{Interior gradient estimates}\label{interior-result}
We first define
\begin{equation}\label{def-omega}
\omega(r):=\sup_{x\in\Omega}\big(\fint_{B_{r}(x)}|a(y)-(a)_{B_r(x)}|^2\ dy\big)^{\frac{1}{2}},
\end{equation}
where 
\begin{equation*}
(a)_{B_r(x)}:=\fint_{B_r(x)}a=\frac{1}{|B_r(x)|}\int_{B_r(x)}a, \quad \forall ~B_r(x)\subset\Omega.
\end{equation*}
 We impose the following assumptions on the coefficient $a(\cdot):=(a_{ij}(\cdot))_{i,j=1}^n$: 
\begin{assumption}\label{main-assump}
We assume that $a$ is defined on $\mathbb R^n$ and satisfies the uniformly elliptic condition with ellipticity constant $\lambda\geq1$:
\begin{equation}\label{condition}
\lambda^{-1}|\eta|^2\leq a_{ij}(x)\eta_i\eta_j,\quad|a_{ij}(x)|\leq\lambda,\quad x\in\mathbb R^n,\quad \forall~ \eta=(\eta_1,...,\eta_n)\in\mathbb R^n. 
\end{equation}
Furthermore, we assume that $\omega$ defined in \eqref{def-omega} satisfies the following condition:
\begin{equation}\label{Dini-1}
\int_0^1\frac{(\omega(r))^{\frac{2}{p}}}{r}\ dr<+\infty\quad \mbox{if}\quad p\geq 2,
\end{equation}
and 
\begin{equation}\label{Dini-2}
\int_0^1\frac{\omega(r)}{r}\ dr<+\infty\quad \mbox{if}\quad 1<p<2.
\end{equation}
\end{assumption}

It is noted that the Dini mean oscillation condition is weaker than the usual Dini continuity condition. The example can be found in \cite{dk2017}. Moreover, the comparison estimates from \cite[(4.35)]{dm2010} and \cite[Lemma 3.4]{dm2011} are significant in the proof of pointwise gradient estimates for the equation \eqref{p-laplace}. However, in the current case under consideration, these estimates are not directly applicable. To overcome these difficulties, we first prove an analogous estimate in Lemma \ref{lem-Dv-Dw}. The novelty here is that the right-hand side of the estimate involves $\|\nabla w\|_{L^\infty}$ instead of $\|\nabla w\|_{L^p}$ for $1<p<\infty$. Here, $w$ is the weak solution of the homogeneous equation in \eqref{Equ w=u}. Then combining with Campanato’s approach introduced in \cite{g1983,l1987} and the iteration technique, we obtain an a priori estimate of $\|\nabla w\|_{L^{\infty}}$ as showed in Lemma \ref{lem nabla w infty}. With these estimates in hand, we are ready to derive our desired results.

Our first main result is stated below.

\begin{theorem}\label{main result Pwg}(Interior pointwise gradient estimate). 
Let $p>1$ and suppose that $u\in$$W^{1,p}_{\text{loc}}(\Omega)$ is a weak solution to \eqref{p-laplace}. Then under Assumption \ref{main-assump}, there exists a constant $C=C(n,p,\lambda,\omega)$ such that for any Lebesgue point $x$ of $\nabla u$, with $B_{R}(x)\subset \Omega$ and $R\in (0,1]$, the following estimates hold:

(i) when $p\geq2$, then 
\begin{equation}\label{p>2 nabla u point}
    |\nabla u(x)|\leq C{\bf W}_{\frac{1}{p},p}^R(|\mu|)(x)+C\fint_{B_R(x)}(|\nabla u(y)|+s)\ dy;
\end{equation}

(ii) when $3/2\leq p<2$, then
\begin{equation}\label{3/2<p<2 nabla u point}
    |\nabla u(x)|\leq C[{\bf I}_1^R(|\mu|)(x)]^{\frac{1}{p-1}}+C\Bigg(\fint_{B_R(x)}(|\nabla u(y)|+s)^{2-p}\ dy\Bigg)^{\frac{1}{2-p}};
\end{equation}

(iii) when $1<p<3/2$, then
\begin{equation}\label{1<p<3/2 nabla u point}
    |\nabla u(x)|\leq C[{\bf I}_1^R(|\mu|)(x)]^{\frac{1}{p-1}}+C\Bigg(\fint_{B_R(x)}(|\nabla u(y)|+s)^{\frac{(p-1)^2}{2}}\ dy\Bigg)^{\frac{2}{(p-1)^2}}.
\end{equation}
\end{theorem}

\begin{remark}
Our pointwise bounds in Theorem \ref{main result Pwg} hold for the range $1<p<\infty$ under the assumption that the modulus of continuity of the coefficient $a(x)$ in the $L^2$-mean sense satisfies the Dini conditions \eqref{Dini-1} and \eqref{Dini-2}, respectively. This condition is weaker than  those in previous works as mentioned above. Investigating the scenario where the modulus of continuity of the coefficient $a(x)$ meets the Dini condition in the $L^1$-mean sense, as explored in the linear case by \cite{dk2017}, is quite challenging. This issue is a subject we plan to address in our future research endeavors.
\end{remark}

The second result is about the continuity of $\nabla u$ as follows.

\begin{theorem}\label{main-result-cont}(Gradient continuity). 
Under the same conditions as in Theorem \ref{main result Pwg}, $\nabla u$ is continuous in $\Omega$ if the following assumptions hold:

(i) the functions 
\begin{equation}\label{assump-Wp}
x\rightarrow{\bf W}_{\frac{1}{p},p}^R(|\mu|)(x)~converge~locally~uniformly~to~zero~in~\Omega~as~R\rightarrow0,\quad \mbox{when}~p\geq2;
\end{equation}

(ii) the functions
\begin{equation}\label{assump Ip}
    x\rightarrow{\bf I}_1^R(|\mu|)(x)~converge~locally~uniformly~to~zero~in~\Omega~as~R\rightarrow 0,\quad \mbox{when}~ 1<p<2.
\end{equation}
\end{theorem}

\begin{remark}\label{rmk-general}
In fact, by using a similar argument, our results in Theorems \ref{main result Pwg} and \ref{main-result-cont} can be applied to the quasilinear elliptic equation \eqref{general-eq} with measure data.
The vector field $A=(A_1,\dots,A_n):\Omega\times\mathbb R^n\rightarrow\mathbb R^n$ satisfies the assumptions as follows: there exist constants $\lambda\geq1$, $s\geq0$, and $p>1$ such that
\begin{equation*}
|A(x,\xi)|\leq \lambda(s^2+|\xi|^2)^{(p-1)/2},\quad|D_{\xi}A(x,\xi)|\leq \lambda(s^2+|\xi|^2)^{(p-2)/2},
\end{equation*}
\begin{equation*}
\langle D_{\xi}A(x,\xi)\eta,\eta\rangle\geq\lambda^{-1}(s^2+|\xi|^2)^{(p-2)/2}|\eta|^2,\quad\forall~\eta=(\eta_1,...,\eta_n)\in\mathbb R^n,
\end{equation*}
and for any $x_0\in\Omega$ and $r\in(0,1)$, there exists a $\bar A_{x_0,r}(\xi)$ depending on $x_0$ and $r$ such that 
\begin{equation*}
\tilde\omega(r):=\sup_{\substack{x_0\in\Omega\\ \xi\in\mathbb R^n\setminus\{0\}}}\Bigg(\fint_{B_r(x_0)}\Big|\frac{A(x,\xi)-\bar A_{x_0,r}(\xi)}{(|\xi|^2+s^2)^{(p-1)/2}}\Big|^2\ dx\Bigg)^{\frac{1}{2}},\quad\forall~B_r(x_0)\subset\Omega
\end{equation*}
satisfies 
\begin{equation*}
\int_0^1\frac{(\tilde\omega(r))^{\frac{2}{p}}}{r}\ dr<+\infty\quad \mbox{if}\quad p\geq 2,
\end{equation*}
and 
\begin{equation*}
\int_0^1\frac{\tilde\omega(r)}{r}\ dr<+\infty\quad \mbox{if}\quad 1<p<2.
\end{equation*}
\end{remark}

\begin{remark}
In the proof of Theorem \ref{main-result-cont}, we establish an a modulus of continuity estimate of the gradient as stated in Theorem \ref{thm continuity}. Kuusi and Mingione \cite[Theorem 1.4,Theorem 1.6]{km2012-2} have demonstrated a similar result under the assumption that $(\tilde\omega(r))^{2/p}$ is Dini–H\"{o}lder of order $\tilde\alpha$ satisfying
\begin{equation*}
\sup_{r}\int_0^r\frac{(\tilde\omega(\rho))^{\frac{2}{p}}}{\rho^{\tilde\alpha}}\ \frac{d\rho}{\rho}<+\infty\quad \mbox{if}\quad p\geq 2,
\end{equation*}
and that $(\tilde\omega(r))^{\sigma}$ is Dini–H\"{o}lder of order $\tilde\alpha$ for some $\sigma\in(0,1)$, satisfying
\begin{equation*}
\sup_{r}\int_0^r\frac{(\tilde\omega(\rho))^{\sigma}}{\rho^{\tilde\alpha}}\ \frac{d\rho}{\rho}<+\infty\quad \mbox{if}\quad 2-1/n<p\leq 2.
\end{equation*}
Here, $\tilde\alpha\in(0,\alpha_M)$ and $\alpha_M\in(0,1]$. In our setting, we allow $\sigma=1$ and $1<p\leq2-1/n$.
\end{remark}

Next we give several consequences of Theorem \ref{main-result-cont}.
We recall that a function $f$ defined in $\Omega$ belongs to the Lorentz space $L^{n,\gamma}$ for $\gamma>0$ if and only if 
\begin{equation*}
\left(\int_0^\infty\big(t^n|\{x\in\Omega: |f(x)|>t\}|\big)^{\gamma/n}\frac{dt}{t}\right)^{1/\gamma}<\infty.
\end{equation*}
In \cite{Kuusi2014}, Kuusi and Mingione proved that the gradient of the solution to the $p$-Laplacian system with $p>1$
\begin{equation*}
-\Div\big(a(x)(|\nabla u|^2+s^2)^{\frac{p-2}{2}}\nabla u\big)=F\quad\mbox{in}~\Omega
\end{equation*}
is continuous by assuming that $a(x)$ is Dini continuous and $F\in L^{n,1}$ locally in $\Omega$. Dong and Zhu \cite{dz2024} derived the continuity of $\nabla u$ to the solution of \eqref{general-eq} with $p\in(1,2)$. By \cite[Lemma 3]{MR2010} and Theorem \ref{main-result-cont}, we obtain the following corollary.
\begin{corollary}\label{coro Lorentz}(Gradient continuity via Lorentz spaces). Let $p\in(1,\infty)$ and $u\in W_{\text{loc}}^{1,p}(\Omega)$ be a solution of \eqref{p-laplace}. Suppose that Assumption \ref{main-assump} is satisfied and $\mu$ meets the following condition:
\begin{equation*}
\mbox{when}~p\geq2, ~\mu\in L^{n,\frac{1}{p-1}} ~\mbox{holds~locally~in~} \Omega
\end{equation*}
and
\begin{equation*}
\mbox{when}~1<p<2, ~\mu\in L^{n,1} ~\mbox{holds~locally~in~} \Omega.
\end{equation*}
Then $\nabla u$ is continuous in $\Omega$.
\end{corollary}

\begin{corollary}(Gradient continuity via density). 
Let $p\in(1,\infty)$ and $u\in W_{\text{loc}}^{1,p}(\Omega)$ be a solution of \eqref{p-laplace}. Suppose that Assumption \ref{main-assump} is satisfied and $\mu$ satisfies 
\begin{equation*}
    |\mu|(B_r(x))\leq Cr^{n-1}g(r)
\end{equation*}
    for any ball $B_r(x)\subset\subset\Omega$, where $C$ is a nonnegative constant and $g:[0,\infty)\rightarrow[0,\infty)$ is a function satisfying the following condition: for some $R>0$, 
    \begin{equation*}
        \int_0^R(g(t))^{\frac{1}{p-1}}\frac{dt}{t}<\infty\quad \mbox{if}\quad p\geq 2
    \end{equation*}
    and 
    \begin{equation*}
        \int_0^Rg(t)\frac{dt}{t}<\infty \quad \mbox{if}\quad 1<p<2.
    \end{equation*}
    Then $\nabla u$ is continuous in $\Omega$.
\end{corollary}

We shall derive the following gradient H\"{o}lder continuity from an a modulus of continuity estimate of the gradient in Theorem \ref{thm continuity}. This generalizes \cite[Corollary 1.7]{dz2024} by extending the range of $p$ from $p\in(1,2)$ to $p\in(1,\infty)$, and improves the result in \cite[Theorem 5.3]{l1993}, who proved $u\in C_{\text{loc}}^{1,\beta_1}(\Omega)$ for some $\beta_1\in(0,1)$ depending on $\beta\in(0,1]$. 

\begin{corollary}\label{coro C^1,n}(Gradient H\"{o}lder continuity via density)
Let $p\in(1,\infty)$ and $u\in W_{loc}^{1,p}(\Omega)$ be a solution of \eqref{p-laplace}. Under the conditions in  Assumption \ref{main-assump}, if there exists a constant $\alpha=\alpha(n, p, \lambda)\in(0,1)$ such that for some constants $C>0$ and $\beta\in(0,\alpha)$, 
    \begin{equation*}
        \omega(r)\leq Cr^{\beta}\quad when \quad r>0
    \end{equation*}
 and for any $B_\rho(x)\subset\subset\Omega$,
 \begin{equation*}
|\mu|(B_\rho(x))\leq  C\rho^{n-1+\beta(p-1)} \quad \mbox{when}~p\geq 2; \quad |\mu|(B_\rho(x))\leq C\rho^{n-1+\beta} \quad \mbox{when}~1<p<2
 \end{equation*}
 hold, then $u\in C_{\text{loc}}^{1,\beta}(\Omega)$. 
\end{corollary}

\subsection{Boundary gradient estimates}\label{boundary-result}
\begin{definition}\label{def chi}
Let $\Omega$ be a domain in $\mathbb R^n$. For any $x_0=(x'_0,x_{0n})\in\partial\Omega$, if there exist a constant $R_0\in (0,1]$, a $C^{1,\text{DMO}}$ function (i.e.,$C^1$ function whose first derivatives are of Dini mean oscillation) $\chi$: $\mathbb R^{n-1}\rightarrow \mathbb R$, and a coordinate system depending on $x_0$ such that for any $r\in(0,R_0)$,
\begin{equation}\label{sup chi-chi}
\varrho_0(r):=\sup_{x'\in B'_{r}(x'_0)}\big(\fint_{B'_r(x'_0)}|\nabla_{x'}\chi(x')-(\nabla _{x'}\chi)_{B'_{r}(x'_0)}|^2\ dx'\big)^{\frac{1}{2}}
\end{equation}
satisfies the Dini condition
\begin{equation*}
    \int_0^1\frac{(\varrho_0(r))^{\frac{2}{p}}}{r}dr<+\infty \quad \mbox{if}\quad p\geq2
\end{equation*}
and
\begin{equation*}
    \int_0^1\frac{\varrho_0(r)}{r}dr<+\infty\quad \mbox{if}\quad 1<p<2,
\end{equation*}
and in the new coordinate system, 
\begin{equation*}
    |\nabla_{x'}\chi(x'_0)|=0, \quad \Omega_{R_0}(x_0)=\{x\in B_{R_0}(x_0): x_n>\chi(x')\},
\end{equation*}
where $\Omega_{R_0}(x_0)=\Omega\cap B_{R_0}(x_0)$, then we say that $\Omega$ has $C^{1,\text{DMO}}$ boundary.
\end{definition}

We additionally define
\begin{equation*}
    \varrho(r):=\sup_{x\in \bar{\Omega}}\Bigg(\fint_{\Omega_r(x)}|a(y)-(a)_{\Omega_r(x)}|^2\ dy \Bigg)^{\frac{1}{2}},
\end{equation*}
where 
\begin{equation*}
    (a)_{\Omega_r(x)}:=\fint_{B_r(x)}a=\frac{1}{|\Omega_r(x)|}\int_{\Omega_r(x)}a, \quad x\in \overline{\Omega}.
\end{equation*}
\begin{assumption}\label{assumption 2}
Assume that $a$ satisfies \eqref{condition} and the following conditions hold:
\begin{equation*}
  \int_0^1\frac{(\varrho(r))^{\frac{2}{p}}}{r}\ dr<+\infty\quad \mbox{if} \quad p\geq 2,  
\end{equation*}
and 
\begin{equation*}
    \int_0^1\frac{\varrho(r)}{r}\ dr<+\infty \quad\mbox{if}\quad 1<p<2.
\end{equation*}
\end{assumption}
We give the global pointwise gradient estimates as follows.

\begin{theorem}(Boundary pointwise gradient estimate)\label{bou pointwise estiamte}
Let $u$ be a solution to \eqref{p-laplace} with a Dirichlet boundary condition $u=0$ on $\partial \Omega$. Assumption \ref{assumption 2} is satisfied and $\Omega$ has a  $C^{1,\text{DMO}}$ boundary  characterized by $R_0$ and $\varrho_0$  as in Definition \ref{def chi}. Then for any Lebesgue point $x\in\Omega$ of the vector-valued function $\nabla u$ and $R\in(0,1]$, there exists a constant $C=C(n,p,\lambda,\varrho,R_0,\varrho_0)$ such that the following estimates hold:
    
    (i) when $p\geq2$, then
    \begin{equation}\label{bou p>2 point}
        |\nabla u(x)|\leq C{\bf W}_{\frac{1}{p},p}^R(|\mu|)(x)+C\fint_{\Omega_R(x)}\big(|\nabla u(y)|+s\big)\ dy;
    \end{equation}
   
    (ii) when $3/2\leq p<2$, then
    \begin{equation*}
        |\nabla u(x)|\leq C[{\bf I}_1^R(|\mu|(x)]^{\frac{1}{p-1}}+C\Bigg(\fint_{\Omega_R(x)}\big(|\nabla u(y)|+s\big)^{2-p}\ dy\Bigg)^{\frac{1}{2-p}};
    \end{equation*}

    (iii) when $1<p<3/2$, then
    \begin{equation*}
        |\nabla u(x)|\leq C[{\bf I}_1^R(|\mu|)(x)]^{\frac{1}{p-1}}+C\Bigg(\fint_{\Omega_R(x)}\big(|\nabla u(y)|+s\big)^{\frac{(p-1)^2}{2}}\ dy\Bigg)^{\frac{2}{(p-1)^2}}.
    \end{equation*}
\end{theorem}

As a consequence of Theorem \ref{bou pointwise estiamte}, we obtain the global Lipschitz estimate when $\Omega$ is bounded.
\begin{corollary}\label{Omega bounded}
    Let $\Omega\subset\mathbb R^n$ be a bounded domain, and under the as assumed in Theorem \ref{bou pointwise estiamte}, there exists a constant $C=C(n, p, \lambda, \varrho, R_0, \varrho_0, \text{diam}(\Omega))$ such that for any $x\in\overline{\Omega}$ and $R\in(0,1]$, the following estimates hold:

    (i) when $p\geq 2$, then
    \begin{equation}\label{Linftyp2}
        \|\nabla u\|_{L^{\infty}(\Omega)}\leq C\|{\bf W}_{1/p,p}^1(|\mu|)\|_{L^{\infty}(\Omega)}+Cs; 
    \end{equation}

    (ii) when $1<p<2$, then 
    \begin{equation}\label{Linfty1p2}
        \|\nabla u\|_{L^{\infty}(\Omega)}\leq C\|{\bf I}_1^1(|\mu|)\|_{L^{\infty}(\Omega)}^{\frac{1}{p-1}}+Cs.
    \end{equation}
\end{corollary}

The corresponding global gradient modulus of continuity estimates similar to Theorem \ref{thm continuity} are presented in Theorem \ref{thm contonuity in+bou}. From Theorem \ref{thm contonuity in+bou}, we can obtain the boundary gradient continuity that parallel Corollaries \ref{coro Lorentz}--\ref{coro C^1,n}. We shall not present the details here. 

The remainder of the paper is organized as follows. In Section \ref{sec-interior}, we provide the proof of Theorem \ref{main result Pwg}. To this end,  we begin by presenting a collection of auxiliary estimates in Subsection \ref{sub-auxiliary-est}. It is important to note that, in contrast to the established results in \cite[(4.35)]{dm2010} and \cite[Lemma 3.4]{dm2011},  the weaker conditions on the coefficients stipulated in Assumption \ref{main-assump} necessitate a different approach. Specifically, the right-hand side of the comparison estimate in Lemma \ref{lem-Dv-Dw} is  $\|\nabla w\|_{L^\infty}$ rather than $\|\nabla w\|_{L^p}$. This distinction requires us to prove an a priori estimate of $\|\nabla w\|_{L^\infty}$, as demonstrated in Lemma \ref{lem nabla w infty}. With these key preliminaries, we borrow an idea in \cite{dz2024} to complete the proof of Theorem \ref{main result Pwg} in Subsection \ref{sec-prf-interior}. The proof of the modulus of continuity estimates for the gradient, as stated in Theorem \ref{main-result-cont}, is presented in Section  \ref{sec-modulus}. Section \ref{sec-boundary} is dedicated to the proof of the boundary pointwise estimates of the gradient in Theorem \ref{bou pointwise estiamte} and Corollary \ref{cor bou}, where we also establish the global gradient modulus of continuity estimates as outlined in Theorem \ref{thm contonuity in+bou}.

\section{Interior pointwise estimates of the gradient}\label{sec-interior}
In this section, we first prove several auxiliary estimates and then give the proof of Theorem \ref{main result Pwg}. 

\subsection{Main auxiliary estimates}\label{sub-auxiliary-est}

For any fixed $x_0\in\Omega$, take $B_{2r}(x_0)\subset\subset\Omega$. Let $u\in W_{\text{loc}}^{1,p}(\Omega)$ be a solution to \eqref{p-laplace}. Assume $w\in u+W^{1,p}_0(B_{2r}(x_0))$  is the unique solution to 
\begin{equation}\label{Equ w=u}
    \begin{cases}
        -\Div(a(x)(|\nabla w|^2+s^2)^{\frac{p-2}{2}}\nabla w)=0 &\quad \mbox{in} \quad B_{2r}(x_0),\\
       w=u &\quad \mbox{on} \quad \partial B_{2r}(x_0).
    \end{cases}
\end{equation}

We shall derive the following estimate.

\begin{lemma}\label{lem-Dv-Dw}
Let $v\in w+W^{1,p}_0(B_r(x_0))$ be the unique solution to the problem
\begin{equation}\label{Equ v=w}
    \begin{cases}
        -\Div((a)_{B_r(x_0)}(|\nabla v|^2+s^2)^{\frac{p-2}{2}}\nabla v)=0 &\quad \mbox{in} \quad B_r(x_0),\\
        v=w&\quad \mbox{on} \quad \partial B_r(x_0),
    \end{cases}
\end{equation}
where $w$ satisfies \eqref{Equ w=u}. Then we have 
\begin{equation}\label{Equ Dv-Dw}
    \fint_{B_r(x_0)}|\nabla v-\nabla w|\ dx\leq C(\omega(r))^{\frac{2}{p}}\big(\|\nabla w\|_{L^{\infty}(B_r(x_0))}+s\big)\quad\mbox{if}\quad p\geq 2,
\end{equation}
where $C$ is a constant that depends only on $n$, $p$, and $\lambda$; and
\begin{equation}\label{Equ Dv-Dw-2}
    \fint_{B_r(x_0)}|\nabla v-\nabla w|^{\gamma_0}\ dx\leq C(\omega(r))^{\gamma_0}\big(\|\nabla w\|_{L^{\infty}(B_r(x_0))}+s\big)^{\gamma_0}\quad\mbox{if}\quad 1<p<2,\quad 0<\gamma_0<1,
\end{equation}
where $C$ is a constant that depends on $n$, $p$, $\lambda$, and $\gamma_0$.
\end{lemma}

\begin{proof}
By testing \eqref{Equ w=u} and \eqref{Equ v=w} with $v-w$, we obtain
\begin{equation}\label{Dv-Dw}
\int_{B_r(x_0)}(a)_{B_r(x_0)}(|\nabla v|^2+s^2)^{\frac{p-2}{2}}\nabla v\cdot\nabla(v-w)\ dx=\int_{B_r(x_0)}a(x)(|\nabla w|^2+s^2)^{\frac{p-2}{2}}\nabla w\cdot\nabla(v-w)\ dx.
\end{equation}
By \cite[(2.2), (2.3)]{dm2010}, \eqref{Dv-Dw}, Hölder's inequality, and Young's inequality, we have
\begin{align*}
&\int_{B_r(x_0)}(|\nabla v|^2+|\nabla w|^2+s^2)^{\frac{p-2}{2}}|\nabla v-\nabla w|^2\ dx\\
&\leq C\int_{B_r(x_0)}\langle(a)_{B_r(x_0)}(|\nabla v|^2+s^2)^{\frac{p-2}{2}}\nabla v-(a)_{B_r(x_0)}(|\nabla w|^2+s^2)^{\frac{p-2}{2}}\nabla w, \nabla v-\nabla w\rangle \ dx\\
&=C\int_{B_r(x_0)}\langle (a(x)-(a)_{B_r(x_0)})(|\nabla w|^2+s^2)^{\frac{p-2}{2}}\nabla w, \nabla v-\nabla w\rangle \ dx\\
&\leq C\big(\int_{B_r(x_0)}|a(x)-(a)_{B_r(x_0)}|^2\big)^{\frac{1}{2}}\big(\int_{B_r(x_0)}(|\nabla w|^2+s^2)^{p-1}|\nabla v-\nabla w|^2dx\big)^\frac{1}{2}\\
&\leq\frac{1}{2}\int_{B_r(x_0)}(|\nabla w|^2+|\nabla v|^2+s^2)^{\frac{p-2}{2}}|\nabla v-\nabla w|^2\ dx\\
&\quad+C\int_{B_r(x_0)}|a(x)-(a)_{B_r(x_0)}|^2\ dx\big(\| \nabla w\|_{L^{\infty}(B_r(x_0))}+s\big)^p.
\end{align*}
This gives 
\begin{equation}\label{est-Dw-Dv}
\int_{B_r(x_0)}(|\nabla v|^2+|\nabla w|^2+s^2)^{\frac{p-2}{2}}|\nabla v-\nabla w|^2\ dx
\leq C\int_{B_r(x_0)}|a(x)-(a)_{B_r(x_0)}|^2\ dx\big(\| \nabla w\|_{L^{\infty}(B_r(x_0))}+s\big)^p.
\end{equation}

Next we intend to complete the proof in two cases.

{\bf Case 1: $p\geq 2$.} Using \eqref{est-Dw-Dv}, we have
\begin{equation*}
\int_{B_r(x_0)}|\nabla v-\nabla w|^p\ dx
\leq C\int_{B_r(x_0)}|a(x)-(a)_{B_r(x_0)}|^2\ dx\big(\| \nabla w\|_{L^{\infty}(B_r(x_0))}+s\big)^p,
\end{equation*}
which combined with Hölder's inequality, we obtain 
\begin{align*}
\int_{B_r(x_0)}|\nabla v-\nabla w|\ dx&\leq Cr^{n(1-\frac{1}{p})}\big(\int_{B_r(x_0)}|\nabla v-\nabla w|^p\ dx\big)^{\frac{1}{p}}\\
&\leq Cr^{n(1-\frac{1}{p})}\big(\int_{B_r(x_0)}|a(x)-(a)_{B_r(x_0)}|^2\ dx\big)^{\frac{1}{p}}\big(\|\nabla w\|_{L^{\infty}(B_r(x_0))}+s\big).
\end{align*}
This directly yields \eqref{Equ Dv-Dw}.

{\bf Case 2: $1<p<2$.} Using a well-known inequality (see, \cite[(2.2)]{dm2010})
\begin{equation}\label{well known equ}
    c^{-1}(|\xi_2|^2+|\xi_2|^2+s^2)^{\frac{p-2}{2}}\leq \frac{|V(\xi_2)-V(\xi_1)|^2}{|\xi_2-\xi_1|^2}\leq c(|\xi_1|^2+|\xi_2|^2+s^2)^{\frac{p-2}{2}}, \quad c>1,
\end{equation}
where
\begin{equation*}
    V(\xi):=(|\xi|^2+s^2)^{\frac{p-2}{4}}\xi, \quad \xi\in\mathbb R^n,
\end{equation*}
we obtain 
\begin{align*}
|\nabla v-\nabla w|&=\big((|\nabla v|^2+|\nabla w|^2+s^2)^{\frac{p-2}{2}}|\nabla v-\nabla w|^2\big)^{\frac{1}{2}}(|\nabla v|^2+|\nabla w|^2+s^2)^{\frac{2-p}{4}}\\
&\leq C|V(\nabla v)-V(\nabla w)|(|\nabla v|^2+|\nabla w|^2+s^2)^{\frac{2-p}{4}}.
\end{align*}
Using \eqref{well known equ} and \eqref{est-Dw-Dv}, we have 
\begin{align}\label{V(nabla v)-V(nabla w)}
\fint_{B_r(x_0)}|V(\nabla v)-V(\nabla w)|^2\ dx&\leq C\fint_{B_r(x_0)}(|\nabla v|^2+|\nabla w|^2+s^2)^{\frac{p-2}{2}}|\nabla v-\nabla w|^2\ dx\nonumber\\
&\leq C\fint_{B_r(x_0)}|a(x)-(a)_{B_r(x_0)}|^2\ dx\big(\| \nabla w\|_{L^{\infty}(B_r(x_0))}+s\big)^p.
\end{align}
Then by  Hölder's inequality and \eqref{V(nabla v)-V(nabla w)}, we obtain 
\begin{align}\label{est-Dv-Dw-Lp}
&\fint_{B_r(x_0)}|\nabla v-\nabla w|^p\ dx\nonumber\\
&\leq C\fint_{B_r(x_0)}|V(\nabla v)-V(\nabla w)|^p(|\nabla v|^2+|\nabla w|^2+s^2)^{\frac{p(2-p)}{4}}\ dx\nonumber\\
&\leq C\big(\fint_{B_r(x_0)}|V(\nabla v)-V(\nabla w)|^2\ dx\big)^{\frac{p}{2}}\big(\fint_{B_r(x_0)}(|\nabla v|^2+|\nabla w|^2+s^2)^{\frac{p}{2}}\ dx\big)^{\frac{2-p}{2}}\nonumber\\
&\leq C\big(\fint_{B_r(x_0)}|a(x)-(a)_{B_r(x_0)}|^2\ dx\big)^{\frac{p}{2}}\big(\| \nabla w\|_{L^{\infty}(B_r(x_0))}+s\big)^{\frac{p^2}{2}}\nonumber\\
&\quad\cdot\big(\fint_{B_r(x_0)}(|\nabla v|^2+|\nabla w|^2+s^2)^{\frac{p}{2}}\ dx\big)^{\frac{2-p}{2}}.
\end{align}
 Recalling \cite[(4.32)]{dm2010}, we have for $1<p<2$,
\begin{equation*}
\fint_{B_r(x_0)}|\nabla v|^p\ dx\leq C\fint_{B_r(x_0)}(|\nabla w|+s)^p\ dx.
\end{equation*}
Then 
\begin{align*}
\fint_{B_r(x_0)}(|\nabla v|^2+|\nabla w|^2+s^2)^{\frac{p}{2}}\ dx\leq C\fint_{B_r(x_0)}(|\nabla w|+s)^p\ dx\leq C\big(\| \nabla w\|_{L^{\infty}(B_r(x_0))}+s\big)^p.
\end{align*}
This in combination with \eqref{est-Dv-Dw-Lp} yields 
\begin{equation*}
\fint_{B_r(x_0)}|\nabla v-\nabla w|^p\ dx\leq C\big(\fint_{B_r(x_0)}|a(x)-(a)_{B_r(x_0)}|^2\ dx\big)^{\frac{p}{2}}\big(\| \nabla w\|_{L^{\infty}(B_r(x_0))}+s\big)^p.
\end{equation*}
Using the previous inequality and the Hölder's inequality, we obtain \eqref{Equ Dv-Dw-2}. Thus, Lemma \ref{lem-Dv-Dw} is proved.
\end{proof}

Now we provide an oscillation estimate for the solution of the following equation
\begin{equation}\label{v-laplace}
    -\Div((a)_{B_{r}(x_0)}(|\nabla v|^2+s^2)^{\frac{p-2}{2}} \nabla v))=0 \quad \mbox{in}~ \Omega.
\end{equation}

\begin{theorem}\label{thm v-BMO}
Let $ v \in W^{1,p}_{\text{loc}} (\Omega)$ be a solution to \eqref{v-laplace}.  There exists a constant $\alpha \in (0,1)$    depending on $n$, $p$, and $\lambda$, such that for every $B_R(x_0)\subset \Omega$ and $\rho \in (0,R),$ we have the following estimates:

${(i)}$ when $p\geq2$, then
    \begin{equation}\label{Equ v-BMO}
       \inf_{\textbf{q}\in \mathbb R^n} \fint_{B_{\rho}(x_0)}|\nabla v-\textbf{q}|\leq C\big (\frac{\rho}{R}\big)^{\alpha}\inf_{\textbf{q}\in \mathbb R^n}  \fint_{B_{R}(x_{0})}|\nabla v-\textbf{q}|,
    \end{equation}  
    where C is a constant depending on $n$, $p$, and $\lambda$;

${(ii)}$ when $1<p<2$, then  for any $0<\gamma_0<1$, 
    \begin{equation}\label{Equ v-Bmo-gamma}
          \inf_{\textbf{q}\in \mathbb R^n} \big(\fint_{B_{\rho}(x_0)}|\nabla v-\textbf{q}|^{\gamma_0}\big)^{\frac{1}{\gamma_0}}\leq C\big (\frac{\rho}{R}\big)^{\alpha}\inf_{\textbf{q}\in \mathbb R^n}  \big(\fint_{B_{R}(x_{0})}|\nabla v-\textbf{q}|^{\gamma_0}\big)^{\frac{1}{\gamma_0}},
    \end{equation}
where $C$ is a constant depending on $n$, $p$, $\lambda$, and $\gamma_0$.
\end{theorem}
\begin{proof}
    The proof of (\ref{Equ v-BMO}) and \eqref{Equ v-Bmo-gamma} follows a similar process as that in \cite[Theorem 2.1]{dz2024}. Thus, we omit the details.
\end{proof}
For any ball $B_r(x)\subset\subset\Omega$ and a function $f\in W_{\text{loc}}^{1,p}(\Omega)$, we define 
\begin{equation*}
\phi_f(x,r)=\inf_{{\bf q}\in\mathbb R^n}\fint_{B_r(x)}|\nabla f-{\bf q}| \quad \text{and} \quad \psi_f(x,r)=\inf_{{\bf q}\in\mathbb R^n}\big(\fint_{B_r(x)}|\nabla f-{\bf q}|^{\gamma_0}\big)^{1/\gamma_0},
\end{equation*}
where $\gamma_0$ $\in$ (0,1).
In the sequel, we shall use $\phi_{\bullet}(x,r)$ and $\psi_{\bullet}(x,r)$ to differentiate between various functions.

\begin{prop}\label{prop-phi-Dw}
Let $w$ be the unique solution to \eqref{Equ w=u} and $\alpha\in(0,1)$ is the constant in Theorem \ref{thm v-BMO}. Then for any $\varepsilon\in(0,1)$, we have 
\begin{align}\label{phi w}
\phi_w(x_0,\varepsilon r)\leq C\varepsilon^{\alpha}\phi_w(x_0,r)+C\varepsilon^{-n}(\omega(r))^{\frac{2}{p}}\big(\|\nabla w\|_{L^{\infty}(B_r(x_0))}+s\big)\quad\mbox{if}\quad p\geq 2,
\end{align}
  where $C$ is a constant depending on $n$, $p$, and $\lambda$; and
\begin{align}\label{psi w}
\psi_w(x_0,\varepsilon r)\leq C\varepsilon^{\alpha}\psi_w(x_0,r)+C\varepsilon^{-n/\gamma_0}\omega(r)\big(\|\nabla w\|_{L^{\infty}(B_r(x_0))}+s\big)\quad\mbox{if}\quad 1<p<2,
\end{align}
for any $0<\gamma_0<1$, where  $C$ is a constant depending on $n$, $p$, $\lambda$, and $\gamma_0$. 
\end{prop}

\begin{proof}
For a ball $B_r(x)\subset \Omega$ and a function $f\in W^{1,p}_{\text{loc}}(\Omega)$, there exist ${\bf q}_{x,r}(f)$ and ${\bf q}_{x,r;\gamma_0}(f)\in\mathbb R^n$ such that
\begin{equation}\label{def-q}
\fint_{B_r(x)}|\nabla f-{\bf q}_{x,r}(f)|=\inf_{{\bf q}\in\mathbb R^n}\fint_{B_r(x)}|\nabla f-{\bf q}| ,
\end{equation}
and 
\begin{equation}\label{def q gamma}
    \big(\fint_{B_r(x)}|\nabla f-{\bf q}_{x,r;\gamma_0}(f)|^{\gamma_0}\big)^{1/\gamma_0}=\inf_{{\bf q}\in \mathbb R^n}\big(\fint_{B_r(x)}|\nabla f-{\bf q}|^{\gamma_0}\big)^{1/\gamma_0}.
\end{equation}
When $p\geq 2$, by the definition of ${\bf q}_{x,r}(\cdot)$ in \eqref{def-q} and \eqref{Equ v-BMO}, and the triangle inequality, we obtain
\begin{align}\label{fint-Dw}
\fint_{B_{\varepsilon r}(x_0)}|\nabla w-{\bf q}_{x_0,\varepsilon r}(w)|&\leq \fint_{B_{\varepsilon r}(x_0)}|\nabla w-{\bf q}_{x_0,\varepsilon r}(v)|\nonumber\\
&\leq C\fint_{B_{\varepsilon r}(x_0)}|\nabla v-{\bf q}_{x_0,\varepsilon r}(v)|+C\fint_{B_{\varepsilon r}(x_0)}|\nabla v-\nabla w|\nonumber\\
&\leq C\varepsilon^{\alpha}\fint_{B_{r}(x_0)}|\nabla v-{\bf q}_{x_0,r}(v)|+C\varepsilon^{-n}\fint_{B_{r}(x_0)}|\nabla v-\nabla w|\nonumber\\
&\leq C\varepsilon^{\alpha}\fint_{B_r(x_0)}|\nabla v-{\bf q}_{x_0,r}(w)|+C\varepsilon^{-n}\fint_{B_r(x_0)}|\nabla v-\nabla w|\nonumber\\
&\leq C\varepsilon^{\alpha}\fint_{B_{r}(x_0)}|\nabla w-{\bf q}_{x_0,r}(w)|+C\varepsilon^{-n}\fint_{B_{r}(x_0)}|\nabla v-\nabla w|.
\end{align}
This together with \eqref{Equ Dv-Dw} implies \eqref{phi w}. Similarly, when $1<p<2$, by \eqref{Equ v-Bmo-gamma}, \eqref{def q gamma}, the triangle inequality, and \eqref{Equ Dv-Dw-2}, we derive \eqref{psi w}. Proposition \ref{prop-phi-Dw} is proved.
\end{proof}

Let $R\in(0,1]$ and $B_R(x_0)\subset\Omega$. Let $\alpha_1\in(0,\alpha)$ and choose $\varepsilon=\varepsilon(n,p,\lambda,\alpha,\alpha_1)>0$ small enough such that 
\begin{equation}\label{alpha1}
C\varepsilon^{\alpha-\alpha_1}<1\quad\mbox{and}\quad \varepsilon^{\alpha_1}<1/4.   
\end{equation}

\begin{lemma}\label{lem-phi-w}
Let $w$ be the unique solution to \eqref{Equ w=u} and let $B_{2r}(x)\subset\subset B_{R}(x_0)\subset\Omega$ with $r\leq R/4$. Then for any $\rho\in(0,r]$, the following assertions hold:

(i) when $p\geq 2$, 
\begin{align}\label{iteration}
\phi_w(x,\rho)\leq C\big(\frac{\rho}{r}\big)^{\alpha_1}\phi_w(x,r)+C\tilde\omega(\rho)\big(\|\nabla w\|_{L^{\infty}(B_{r}(x))}+s\big),
\end{align}
and 
\begin{align}\label{iteration-sum}
\sum_{j=0}^{\infty}\phi_w(x,\varepsilon^j\rho)\leq C\big(\frac{\rho}{r}\big)^{\alpha_1}\phi_w(x,r)+C\big(\|\nabla w\|_{L^{\infty}(B_{r}(x))}+s\big)\int_{0}^{\rho}\frac{\tilde\omega(t)}{t}\ dt,
\end{align}
where $C$ is a constant depending on $n$, $p$, and $\lambda$, and $\tilde\omega(\cdot)$ is defined in \eqref{def-tilde-omega};

(ii) when $1<p<2$, for any $0<\gamma_0<1$, 
\begin{equation}\label{psi w iteration}
    \psi_w(x,\rho)\leq C\big(\frac{\rho}{r}\big)^{\alpha_1}\psi_w(x,r)+C\overline{w}(\rho)\big(\|\nabla w\|_{L^{\infty}(B_r(x))}+s\big),
\end{equation}
and
\begin{equation}\label{psi w iteration sum}
    \sum_{j=0}^{\infty}\psi_w(x,\varepsilon^j\rho)\leq C\big(\frac{\rho}{r}\big)^{\alpha_1}\psi_w(x,r)+C\big(\|\nabla w\|_{L^{\infty}(B_r(x))}+s\big)\int_0^{\rho}\frac{\overline{w}(t)}{t}\ dt,
\end{equation}
where $C$ is a constant depending on $n$, $p$, $\lambda$, and $\gamma_0$, and $\overline{\omega}(\cdot)$ is defined in \eqref{def overline omega}.
\end{lemma}

\begin{proof}
{\bf Case 1: $p\geq 2$.}
 For any $B_{2r}(x)\subset\subset B_{R}(x_0)$, by \eqref{phi w} and \eqref{alpha1}, we have
\begin{align*}
\phi_w(x,\varepsilon r)\leq \varepsilon^{\alpha_1}\phi_w(x,r)+C(\omega(r))^{\frac{2}{p}}\big(\|\nabla w\|_{L^{\infty}(B_r(x))}+s\big).
\end{align*}
Then by iteration, for any $B_{2r}(x)\subset\subset B_{R}(x_0)$ with $r\in(0,R/4)$, we get
\begin{align}\label{phi-j}
\phi_w(x,\varepsilon^j r)&\leq \varepsilon^{\alpha_1 j}\phi_w(x,r)+C\sum_{i=1}^{j}\varepsilon^{\alpha_1(i-1)}(\omega(\varepsilon^{j-i}r))^{\frac{2}{p}}\big(\|\nabla w\|_{L^{\infty}(B_{r}(x))}+s\big)\nonumber\\
&\leq \varepsilon^{\alpha_1 j}\phi_w(x,r)+C\tilde\omega(\varepsilon^{j}r)\big(\|\nabla w\|_{L^{\infty}(B_{r}(x))}+s\big),
\end{align}
where
\begin{equation}\label{def-tilde-omega}
\tilde\omega(t):=\sum_{i=0}^{\infty}\varepsilon^{\alpha_1 i}\big((\omega(\varepsilon^{-i}t))^{\frac{2}{p}}[\varepsilon^{-i}t\leq R/2]+(\omega(R/2))^{\frac{2}{p}}[\varepsilon^{-i}t>R/2]\big).
\end{equation}
This employs the Iverson bracket notation, that is, $[P]=1$ if $P$ is true and $[P]=0$ otherwise. Similar to the proof of \cite[Lemma 1]{d2012}, and using (\ref{Dini-1}), we get 
\begin{equation}\label{Dini-cond}
\int_0^1\frac{\tilde\omega(r)}{r}\ dr<+\infty.
\end{equation}
For any given $\rho\in(0,r]$, let $j$ be a nonnegative integer such that $\varepsilon^{j+1}<\frac{\rho}{r}\leq\varepsilon^j$. By \eqref{phi-j} with $\varepsilon^{-j}\rho$ in place of $r$, we obtain
\begin{align*}
\phi_w(x,\rho)&\leq\varepsilon^{\alpha_1 j}\phi_w(x,\varepsilon^{-j}\rho)+C\tilde\omega(\rho)\big(\|\nabla w\|_{L^{\infty}(B_{\varepsilon^{-j}\rho}(x))}+s\big)\nonumber\\
&\leq C\big(\frac{\rho}{r}\big)^{\alpha_1}\phi_w(x,r)+C\tilde\omega(\rho)\big(\|\nabla w\|_{L^{\infty}(B_{r}(x))}+s\big).
\end{align*}
This gives \eqref{iteration}. Replacing $\rho$ with $\varepsilon^{j}\rho$ in \eqref{iteration} and summing in $j$, we have
\begin{align*}
\sum_{j=0}^{\infty}\phi_w(x,\varepsilon^j\rho)\leq C\big(\frac{\rho}{r}\big)^{\alpha_1}\phi_w(x,r)+C\sum_{j=1}^{\infty}\tilde\omega(\varepsilon^j\rho)\big(\|\nabla w\|_{L^{\infty}(B_{r}(x))}+s\big).
\end{align*}
By \cite[Lemma 2.7]{dk2017}, we derive \eqref{iteration-sum}.

{\bf Case 2: $1<p<2$.} From \eqref{psi w} and \eqref{alpha1}, for any $B_{2r}(x)\subset\subset B_R(x_0)$, we have 
\begin{equation*}
    \psi_w(x,\varepsilon r)\leq \varepsilon^{\alpha_1}\psi_w(x,r)+C\omega(r)\big(\|\nabla w\|_{L^{\infty}(B_r(x))}+s\big).
\end{equation*}
By iteration, for any $B_{2r}(x)\subset\subset B_R(x_0)$ with $r\in (0,R/4)$, we obtain
\begin{align}\label{psi w var j}
    \psi_w(x,\varepsilon^j r)&\leq \varepsilon^{\alpha_1j}\psi_w(x,r)+C\sum_{i=1}^j\varepsilon^{\alpha_1(i-1)}\omega(\varepsilon^{j-i}r)\big(\|\nabla w\|_{L^{\infty}(B_r(x))}+s\big)\nonumber\\
    &\leq \varepsilon^{\alpha_1j}\psi_w(x,r)+C\overline{\omega}(\varepsilon^jr)\big(\|\nabla w\|_{L^{\infty}(B_r(x))}+s\big),
\end{align}
where
\begin{equation}\label{def overline omega}
    \overline{\omega}(t):=\sum_{i=0}^{\infty}\varepsilon^{\alpha_1i}\big(\omega(\varepsilon^{-i}t)[\varepsilon^{-i}t\leq R/2]+\omega(R/2)[\varepsilon^{-i}t>R/2]\big).
\end{equation}
Once again, the Iverson bracket notation is utilized here. We also have
\begin{equation}\label{overline omega}
    \int_0^1\frac{\overline{\omega}(t)}{r}\ dt<+\infty.
\end{equation}
Following a proof analogous to that of \eqref{iteration} and \eqref{iteration-sum}, we then derive \eqref{psi w iteration} and \eqref{psi w iteration sum} from \eqref{psi w var j} and \eqref{overline omega}.
\end{proof}

\begin{lemma}\label{lem nabla w infty}
Let $w$ be the unique solution to \eqref{Equ w=u}, we have 
\begin{equation}\label{Dw-infty}
\|\nabla w\|_{L^{\infty}(B_{R/2}(x_0))}\leq C R^{-n}\||\nabla w|+s\|_{L^{1}(B_{R}(x_0))}\quad\mbox{if}\quad p\geq 2,
\end{equation}
where $C$ is a constant depending on $n$, $p$, and $\lambda$; and
\begin{equation}\label{p<2 nable w L infty}
\|\nabla w\|_{L^{\infty}(B_{R/2}(x_0))}\leq CR^{-n/\gamma_0}\| |\nabla w|+s\|_{L^{\gamma_0}(B_R(x_0))}\quad\mbox{if}\quad 1<p<2,\quad 0<\gamma_0<1,
\end{equation}
where $C$ depends on $n$, $p$, $\lambda$, and $\gamma_0$.
\end{lemma}

\begin{proof}
For any given point $x=x_0\in\Omega$, we assume that $B_R(x_0)\subset\Omega$ and $R\in(0,1]$. We shall first prove an a priori estimate under the assumption that $w\in C^{0,1}(\overline{B_R(x_0)})$. Then we will deal with the general case by using approximation.

{\bf Step 1}: The case when $w\in C^{0,1}(\overline{B_R(x_0)})$. 

{\bf Case 1: $p\geq 2$}.  Notice that
\begin{equation*}
|{\bf q}_{x,\varepsilon\rho}(w)-{\bf q}_{x,\rho}(w)|\leq|\nabla w-{\bf q}_{x,\varepsilon\rho}(w)|+|\nabla w-{\bf q}_{x,\rho}(w)|,
\end{equation*}
then by taking the average over $z\in B_{\varepsilon\rho}(x)$, we obtain 
\begin{equation*}
|{\bf q}_{x,\varepsilon\rho}(w)-{\bf q}_{x,\rho}(w)|\leq C\phi_w(x,\varepsilon\rho)+C\phi_w(x,\rho).
\end{equation*}
By iteration, we have 
\begin{equation*}
|{\bf q}_{x,\varepsilon^j\rho}(w)-{\bf q}_{x,\rho}(w)|\leq C\sum_{i=0}^j\phi_w(x,\varepsilon^i\rho).
\end{equation*}
Recalling \eqref{phi-j}, we have 
\begin{equation*}
\lim_{j\rightarrow\infty}\phi_w(x,\varepsilon^j\rho)=0.
\end{equation*}
This together with
\begin{equation*}
\phi_w(x,\varepsilon^j\rho)=\fint_{B_{\varepsilon^j\rho}(x)}|\nabla w-{\bf q}_{x,\varepsilon^j\rho}(w)|
\end{equation*}
implies
\begin{equation*}
\lim_{j\rightarrow\infty}{\bf q}_{x,\varepsilon^j\rho}(w)=\nabla w(x),\quad a.e.\quad x\in \Omega.
\end{equation*}
Thus, for $a.e.$ $x\in\Omega$ with respect to $\nabla w$, we get
\begin{equation}\label{nabla-q w}
|\nabla w(x)-{\bf q}_{x,\rho}(w)|\leq C\sum_{j=0}^{\infty}\phi_w(x,\varepsilon^j\rho).
\end{equation}
By using \eqref{iteration-sum} with $r=\rho$, we obtain, for any $B_{2\rho}(x)\subset\subset B_{R}(x_0)$ with $\rho\in(0,R/4]$,
\begin{equation*}
|\nabla w(x)-{\bf q}_{x,\rho}(w)|\leq C\phi_w(x,\rho)+C\big(\|\nabla w\|_{L^{\infty}(B_{\rho}(x))}+s\big)\int_{0}^{\rho}\frac{\tilde\omega(t)}{t}\ dt.
\end{equation*}
Note that
\begin{align*}
|{\bf q}_{x,\rho}(w)|&\leq \fint_{B_{\rho}(x)}\big(|\nabla w-{\bf q}_{x,\rho}(w)|+|\nabla w|\big)\\
&\leq C\phi_w(x,\rho)+C\rho^{-n}\|\nabla w\|_{L^1(B_{\rho}(x))}\leq C\rho^{-n}\|\nabla w\|_{L^1(B_{\rho}(x))},
\end{align*}
we thus obtain
\begin{equation*}
|\nabla w(x)|\leq C\rho^{-n}\|\nabla w\|_{L^1(B_{\rho}(x))}+C_0\int_{0}^{\rho}\frac{\tilde\omega(t)}{t}\ dt\big(\|\nabla w\|_{L^{\infty}(B_{\rho}(x))}+s\big).
\end{equation*}
Since $\tilde\omega$ satisfies \eqref{Dini-cond}, we can find a constant $\rho_0=\rho_0(n,p,\lambda,\omega,\alpha_1,R)\in(0,R/4]$ small enough such that 
\begin{equation*}
C_0\int_{0}^{\rho_0}\frac{\tilde\omega(t)}{t}\ dt\leq 3^{-1-n}.
\end{equation*}
Then for any $B_{2\rho}(x)\subset\subset B_{R}(x_0)$ with $\rho\in(0,\rho_0]$, 
\begin{equation}\label{Dw-est}
|\nabla w(x)|\leq C\rho^{-n}\|\nabla w\|_{L^1(B_{\rho}(x))}+3^{-n}\big(\|\nabla w\|_{L^{\infty}(B_{\rho}(x))}+s\big).
\end{equation}

For $k\geq1$, denote $\rho_k=(1-2^{-k})R$. Taking $k_0$ large enough such that $2^{-k_0-1}\leq \rho_0$. By using \eqref{Dw-est} with $\rho=2^{-k-1}R$, for any $x\in B_{\rho_k}(x_0)$ and $k\geq k_0$, we have 
\begin{align*}
\|\nabla w\|_{L^\infty(B_{\rho_k}(x_0))}+s\leq C(2^{-k-1}R)^{-n}\|\nabla w\|_{L^1(B_{\rho_{k+1}}(x_0))}+3^{-n}\big(\|\nabla w\|_{L^{\infty}(B_{\rho_{k+1}}(x_0))}+s\big)+s.
\end{align*}
Now we multiply it by $3^{-nk}$ and sum the terms with respect to $k=k_0, k_{0}+1, \dots$ to obtain
\begin{align*}
\sum_{k=k_0}^{\infty}3^{-nk}\big(\|\nabla w\|_{L^\infty(B_{\rho_k}(x_0))}+s\big)\leq CR^{-n}\|\nabla w\|_{L^1(B_{R}(x))}+\sum_{k=k_0+1}^{\infty}3^{-nk}\big(\|\nabla w\|_{L^{\infty}(B_{\rho_{k+1}}(x_0))}+s\big)+Cs.
\end{align*}
This implies 
\begin{equation}\label{step1 nabla w L infty}
\|\nabla w\|_{L^\infty(B_{R/2}(x_0))}+s\leq CR^{-n}\||\nabla w|+s\|_{L^1(B_{R}(x_0))}.
\end{equation}

{\bf Case 2: $1<p<2$.} Following a proof similar to that of \eqref{step1 nabla w L infty}, from \eqref{psi w var j}, \eqref{psi w iteration sum}, and \eqref{overline omega}, we obtain
\begin{equation}\label{step1 case 2 w}
    \|\nabla w\|_{L^{\infty}(B_{R/2}(x_0))}\leq CR^{-n/\gamma_0}\||\nabla w|+s\|_{L^{\gamma_0}(B_R(x_0))}.
\end{equation}

{\bf Step 2}: The general situation.

{\bf Case 1: $p\geq 2$.}
Let $\varphi\geq0$ be a infinitely differentiable function with unit integral supported in $B_1(0)$. Define $\varphi_m(\cdot)=m^d\varphi(m\cdot)$ such that
\begin{equation*}
    \varphi_m\in C^{\infty}_0(B_{\frac{1}{m}}(0)), \quad \varphi _m\geq 0, \quad \int_{B_{\frac{1}{m}}(0)}\varphi_m=1,
\end{equation*}
where $m$ is a positive integer. Subsequently, we choose $R_1\in(0,R) $, $R_2\in (R_1,R)$, and then mollify $a(\cdot)$ by constructing
\begin{equation*}
     a_m(x)=(a(\cdot)*\varphi_m)(x),\quad x\in B_{R_2}(x_0).
\end{equation*}
It is readily seen that $a_m(\cdot)$ is also well defined and satisfies conditions (\ref{condition}) and \eqref{Dini-1} in $B_{R_2}(x_0)$, where $m>\frac{1}{R-R_2}$.

Now we let $w_m\in w+W^{1,p}_0(B_{R_2}(x_0))$ be the unique solution to
\begin{equation}\label{Equ um=u}
\begin{cases}
    -div(a_m(x)(|\nabla w_m|^2+s^2)^{\frac{p-2}{2}}\nabla w_m)=0\quad &in \quad B_{R_2}(x_0),\\w_m=w\quad &on \quad\partial B_{R_2}(x_0).
\end{cases}
\end{equation}
Taking $w_m-w$ as a test function for \eqref{Equ um=u}, we obtain
\begin{align}\label{Equ a_m}
    \int_{B_{R_2}(x_0)}\langle a_m(x)(|\nabla w_m|^2+s^2)^{\frac{p-2}{2}}\nabla w_m, \nabla w_m\rangle dx
    =\int _{B_{R_2}(x_0)}\langle a_m(x)(|\nabla w_m|^2+s^2)^{\frac{p-2}{2}}\nabla w_m, \nabla w\rangle dx.
\end{align}
By applying (\ref{condition}), the following is derived: 
\begin{align*}
    \int_{B_{R_2}(x_0)}\langle a_m(x)(|\nabla w_m|^2+s^2)^{\frac{p-2}{2}}\nabla w_m, \nabla w_m\rangle dx
    \geq \lambda^{-1}\int_{B_{R_2}(x_0)}|\nabla w_m|^pdx.
\end{align*}
    The right-hand side of (\ref{Equ a_m}) can be bounded, by (\ref{condition}),  and Young's inequality, as follows
    \begin{align*}
       &\int_{B_{R_2}(x_0)}\langle a_m(x)(|\nabla w_m|^2+s^2)^{\frac{p-2}{2}}\nabla w_m, \nabla w\rangle dx\\
        &\leq \lambda \int_{B_{R_2}(x_0)}(|\nabla w_m|^2+s^2)^{\frac{p-1}{2}}|\nabla w|dx\\
       &\leq \frac{1}{2\lambda}\int_{B_{R_2}(x_0)}(|\nabla w_m|+s)^pdx+C\int_{B_{R_2}(x_0)}|\nabla w|^pdx.
    \end{align*}
    Therefore, (\ref{Equ a_m}) indicates that
    \begin{equation}\label{inequ Dum}
       \|\nabla w_m\|^p_{L^p(B_{R_2}(x_0))} \leq C\| |\nabla w|+s\|^p_{L^p(B_{R_2}(x_0))},
    \end{equation}
    where $C$ is a constant independent of $m$.
Using \eqref{condition} and \eqref{well known equ} yield
\begin{equation*}
    c_0^{-1}|V(\xi_2)-V(\xi_1)|^2\leq \langle(a(x)(|\xi_2|^2+s^2)^{\frac{p-2}{2}}\xi_2-a(x)(|\xi_1|^2+s^2)^{\frac{p-2}{2}}\xi_1, \xi_2-\xi_1\rangle,
\end{equation*}
where $c_0=c_0(n,p,\lambda)$ is a positive constant. We observe that this inequality holds for $a_m$ as well.
Taking $(w_m-w)1_{B_{R_2}(x_0)}$ as the test function for the equation in \eqref{Equ w=u} and \eqref{Equ um=u}, we obtain
\begin{align*}
    &\int_{B_{R_2}(x_0)}|V(\nabla w_m)-V(\nabla w)|^2dx\\
    &\leq c_0\int_{B_{R_2}(x_0))}\langle a_m(x)(|\nabla w_m|^2+s^2)^{\frac{p-2}{2}}\nabla w_m-a_m(x)(|\nabla u|^2+s^2)^{\frac{p-2}{2}}\nabla w, \nabla w_m-\nabla w\rangle dx\\
    &=\int_{B_{R_2}(x_0)}\langle a(x)(|\nabla w|^2+s^2)^{\frac{p-2}{2}}\nabla w-a_m(|\nabla w|^2+s^2)^{\frac{p-2}{2}}\nabla w, \nabla w_m-\nabla w\rangle dx\\
    &\leq \|((a(x)-a_m(x))(|\nabla w|^2+s^2)^{\frac{p-2}{2}}\nabla w\|_{L^{\frac{p}{p-1}}(B_{R_2}(x_0))}\|w_m-w\|_{W^{1,p}(B_{R_2}(x_0))}.
\end{align*}
Combining the fact that $a_m$ converges to $a$ almost everywhere with (\ref{inequ Dum}), we obtain
\begin{equation}\label{V-V rightarrow}
    \int_{B_{R_2}(x_0)}|V(\nabla w_m)-V(\nabla w)|^2dx\rightarrow 0\quad\mbox{as}~m\rightarrow\infty.
\end{equation}
It is obvious that
\begin{equation*}
    \int_{B_{R_2}(x_0)}|\nabla w_m-\nabla w|^p\ dx\leq C\int_{B_{R_2}(x_0)}(|\nabla w_m|^2+|\nabla w|^2+s^2)^{\frac{p-2}{2}}|\nabla w_m-\nabla w|^2\ dx.
\end{equation*}
Using the above in conjunction with \eqref{well known equ} and \eqref{V-V rightarrow}, we get
\begin{equation*}
    \nabla w_m\rightarrow \nabla w \quad in\quad L^p(B_{R_2}(x_0)).
\end{equation*}
This implies that there exists a subsequence $\{m_j\}$ such that $\nabla w_{m_j}\rightarrow \nabla w$  almost everywhere in $B_{R_2}(x_0)$.
Given the smoothness and regularity of $a_m$ (as referenced in \cite{d1984}), we find that $w_m\in C^{1,\alpha}_{\text{loc}}(B_{R_2}(x_0))$. Hence, the Lipschitz estimate for $u_m$ in $B_{R_1/2}(x_0)$ as stated in \eqref{step1 nabla w L infty} is also valid. That is,
\begin{equation*}
    \| \nabla w_m \|_{L^{\infty}(B_{R_1/2}(x_0))}\leq C R_1^{-n}\|\ |\nabla w_m|+s\|_{L^1(B_{R_1}(x_0))}.
\end{equation*}
Taking $m=m_j\nearrow \infty$ and $R_1\nearrow R$, we obtain the Lipschitz estimate \eqref{Dw-infty} in the neighborhood of any given point $x=x_0$.

{\bf Case 2: $1<p<2$.} By a similar argument as in the proof of \cite[Theorem 1.3]{dz2024} and using \eqref{step1 case 2 w}, we arrive at \eqref{p<2 nable w L infty}. 
\end{proof}

\subsection{Proof of Theorem \ref{main result Pwg}}\label{sec-prf-interior}
Before proceeding the proof, we first recall the following technical lemma, which generalizes similar results in \cite{np2019,np2022}. 

\begin{lemma}\label{lem-comparison}
Let $w$ be the unique solution to \eqref{Equ w=u}. Then,

\textbf{(i)} when $p\geq2$, there exists a constant $C>1$ depending only $n$, $p$ and $\lambda$ such that 
\begin{equation}\label{nabla u-nabla w}
\fint_{B_{2r}(x_0)}|\nabla u-\nabla w|\ dx\leq C\Bigg(\frac{|\mu|(B_{2r}(x_0))}{r^{n-1}}\Bigg)^{\frac{1}{p-1}};
\end{equation}

\textbf{(ii)} when $3/2\leq p<2$, take $\gamma_0=2-p$, there exists a constant $C>1$, depending solely on $n$, $p$, $\lambda$ and $\gamma_0$, such that
\begin{align}\label{nabla u-nabla w gamma}
&\Bigg(\fint_{B_{2r}(x_0)}|\nabla u-\nabla w|^{\gamma_0}\ dx\Bigg)^{\frac{1}{\gamma_0}}\nonumber\\&\leq C\Bigg(\frac{|\mu|(B_{2r}(x_0))}{r^{n-1}}\Bigg)^{\frac{1}{p-1}}+C\frac{|\mu|(B_{2r}(x_0))}{r^{n-1}}\fint_{B_{2r}(x_0)}(|\nabla u|+s)^{2-p}\ dx;
\end{align}

\textbf{(iii)} when $1<p<3/2$, take $\gamma_0=(p-1)^2/2$, there exists a constant $C>1$, depending on $n$, $p$, and $\lambda$, such that 
\begin{align}\label{nabla u-nabla w gamma02}
&\Bigg(\fint_{B_{2r}(x_0)}|\nabla u-\nabla w|^{\gamma_0}\ dx\Bigg)^{\frac{1}{\gamma_0}}\nonumber\\&\leq C\Bigg(\frac{|\mu|(B_{2r}(x_0))}{r^{n-1}}\Bigg)^{\frac{1}{p-1}}+C\frac{|\mu|(B_{2r}(x_0))}{r^{n-1}}\left(\fint_{B_{2r}(x_0)}(|\nabla u|+s)^{\gamma_0}\ dx\right)^{\frac{2-p}{\gamma_0}}.
\end{align}
\end{lemma}

\begin{proof}
The case of $p\geq2$ follows from \cite[Lemma 3.3]{dm2011}. For $\frac{3}{2}\leq p<2$, the arguments in \cite[Lemma 3.2]{dz2024} also work in our case. The case of $1<p<\frac{3}{2}$ is similar to that in \cite[Theorem 1.2]{np2023}. Thus, we omit the details.
\end{proof}

\begin{prop}\label{prop-phi-Du}
Let $u\in W_{\text{loc}}^{1,p}(\Omega)$ be a solution to \eqref{p-laplace} and  $\alpha\in(0,1)$ is the constant in Theorem \ref{thm v-BMO}. Then for any $\varepsilon\in(0,1)$ and $B_{2r}(x_0)\subset\Omega$, we have the following assertions:

\textbf{(i)} when $p\geq2$,
\begin{align}\label{ite-phi-u}
\phi_u(x_0,\varepsilon r)&\leq C\varepsilon^{\alpha}\phi_u(x_0,r)+C\varepsilon^{-n}\Bigg(\frac{|\mu|(B_{2r}(x_0))}{r^{n-1}}\Bigg)^{\frac{1}{p-1}}\nonumber\\
&\quad+C\varepsilon^{-n}(\omega(r))^{\frac{2}{p}}\fint_{B_{2r}(x_0)}\big(|\nabla u|+s\big)\ dx,
\end{align}
where $C$ is a constant depending on $n$, $p$, and $\lambda$; 

\textbf{(ii)} when $3/2\leq p<2$, take $\gamma_0=2-p$,
\begin{align}\label{p<2 psi u var r}
\psi_u(x_0,\varepsilon r)&\leq C\varepsilon^{\alpha}\psi_u(x_0,r)+C\varepsilon^{-\frac{n}{\gamma_0}}\Bigg(\frac{|\mu|(B_{2r}(x_0))}{r^{n-1}}\Bigg)^{\frac{1}{p-1}}\nonumber\\
&\quad+C\varepsilon^{-\frac{n}{\gamma_0}}\frac{|\mu|(B_{2r}(x_0))}{r^{n-1}}\fint_{B_{2r}(x_0)}\big(|\nabla u|+s\big)^{2-p}\ dx\nonumber\\
&\quad+C\varepsilon^{-\frac{n}{\gamma_0}}\omega(r)\Bigg(\fint_{B_{2r}(x_0)}\big(|\nabla u|+s\big)^{2-p}\ dx\Bigg)^{\frac{1}{2-p}},
\end{align}
where $C$ is a constant depending on $n$, $p$, $\lambda$, and $\gamma_0$;

\textbf{(iii)} when $1<p<3/2$, for $\gamma_0=(p-1)^2/2$,
\begin{align}\label{p<3/2 psi u var r}
\psi_u(x_0,\varepsilon r)&\leq C\varepsilon^{\alpha}\psi_u(x_0,r)+C\varepsilon^{-\frac{n}{\gamma_0}}\Bigg(\frac{|\mu|(B_{2r}(x_0))}{r^{n-1}}\Bigg)^{\frac{1}{p-1}}\nonumber\\
 &\quad +C\varepsilon^{-\frac{n}{\gamma_0}}\frac{|\mu|(B_{2r}(x_0))}{r^{n-1}}\Bigg(\fint_{B_{2r}(x_0)}\big(|\nabla u|+s\big)^{\gamma_0}\ dx\Bigg)^{\frac{2-p}{\gamma_0}}\nonumber\\
    &\quad +C\varepsilon^{-\frac{n}{\gamma_0}}\omega(r)\Bigg(\fint_{B_{2r}(x_0)}\big(|\nabla u|+s\big)^{\gamma_0}\ dx\Bigg)^{\frac{1}{\gamma_0}},
\end{align}
where $C$ is a constant depending on $n$, $p$, and $\lambda$.
\end{prop}

\begin{proof}
{\bf Case 1: $p\geq 2$.} Similar to \eqref{fint-Dw}, we have 
\begin{align}\label{Du-q}
\fint_{B_{\varepsilon r}(x_0)}|\nabla u-{\bf q}_{x_0,\varepsilon r}(u)|\leq C\varepsilon^{\alpha}\fint_{B_{r}(x_0)}|\nabla u-{\bf q}_{x_0,r}(u)|+C\varepsilon^{-n}\fint_{B_{r}(x_0)}|\nabla u-\nabla v|.
\end{align}
By the triangle inequality, \eqref{Equ Dv-Dw}, and \eqref{Dw-infty} with $2r$ in place of $R$, we have 
\begin{align*}
\fint_{B_{r}(x_0)}|\nabla u-\nabla v|&\leq \fint_{B_{r}(x_0)}|\nabla u-\nabla w|+\fint_{B_{r}(x_0)}|\nabla v-\nabla w|\\
&\leq \fint_{B_{r}(x_0)}|\nabla u-\nabla w|+C(\omega(r))^{\frac{2}{p}}\big(\|\nabla w\|_{L^{\infty}(B_r(x_0))}+s\big)\\
&\leq \fint_{B_{r}(x_0)}|\nabla u-\nabla w|+C(\omega(r))^{\frac{2}{p}}\fint_{B_{2r}(x_0)}\big(|\nabla w|+s\big)\\
&\leq C\fint_{B_{2r}(x_0)}|\nabla u-\nabla w|+C(\omega(r))^{\frac{2}{p}}\fint_{B_{2r}(x_0)}\big(|\nabla u|+s\big).
\end{align*}
Substituting it into \eqref{Du-q} and using \eqref{nabla u-nabla w}, we obtain
\begin{align*}
\fint_{B_{\varepsilon r}(x_0)}|\nabla u-{\bf q}_{x_0,\varepsilon r}(u)|&\leq C\varepsilon^{\alpha}\fint_{B_{r}(x_0)}|\nabla u-{\bf q}_{x_0,r}(u)|+C\varepsilon^{-n}\fint_{B_{2r}(x_0)}|\nabla u-\nabla w|\\
&\quad+C\varepsilon^{-n}(\omega(r))^{\frac{2}{p}}\fint_{B_{2r}(x_0)}\big(|\nabla u|+s\big)\\
&\leq C\varepsilon^{\alpha}\fint_{B_{r}(x_0)}|\nabla u-{\bf q}_{x_0,r}(u)|+C\varepsilon^{-n}\big(\frac{|\mu|(B_{2r}(x_0))}{r^{n-1}}\big)^{\frac{1}{p-1}}\\
&\quad+C\varepsilon^{-n}(\omega(r))^{\frac{2}{p}}\fint_{B_{2r}(x_0)}\big(|\nabla u|+s\big).
\end{align*}
This finishes the proof of \eqref{ite-phi-u}.

{\bf Case 2: $3/2\leq p<2$.} Similar to \eqref{Du-q}, we have
\begin{align}\label{var nable u-q gam}
    \Bigg(\fint_{B_{\varepsilon r}(x_0)}|\nabla u-{\bf q}_{x_0,\varepsilon r;\gamma_0}(u)|^{\gamma_0}\Bigg)^{\frac{1}{\gamma_0}}&\leq C \varepsilon^{\alpha}\Bigg(\fint_{B_r(x_0)}|\nabla u-{\bf q}_{x_0,r;\gamma_0}(u)|^{\gamma_0}\Bigg)^{\frac{1}{\gamma_0}}\nonumber\\&\quad+C\varepsilon^{-\frac{n}{\gamma_0}}\Bigg(\fint_{B_r(x_0)}|\nabla u-\nabla v|^{\gamma_0}\Bigg)^{\frac{1}{\gamma_0}}.
\end{align}
Following the proof above, using the triangle inequality and replacing $R$ with $2r$ in \eqref{Equ Dv-Dw-2} and \eqref{p<2 nable w L infty}, we derive
\begin{align*}
    \fint_{B_r(x_0)}|\nabla u-\nabla v|^{\gamma_0}&\leq \fint_{B_r(x_0)}|\nabla u-\nabla w|^{\gamma_0}+\fint_{B_r(x_0)}|\nabla v-\nabla w|^{\gamma_0}\\
    &\leq \fint_{B_r(x_0)}|\nabla u-\nabla w|^{\gamma_0}+C(\omega(r))^{\gamma_0}\big(\|\nabla w\|_{L^{\infty}(B_r(x_0))}+s\big)^{\gamma_0}\\
    &\leq \fint_{B_r(x_0)}|\nabla u-\nabla w|^{\gamma_0}+C (\omega(r))^{\gamma_0}\fint_{B_{2r}(x_0)}\big(|\nabla w|+s\big)^{\gamma_0}\\
    &\leq \fint_{B_{2r}(x_0)}|\nabla u-\nabla w|^{\gamma_0}+C(\omega(r))^{\gamma_0}\fint_{B_{2r}(x_0)}\big(|\nabla u+s\big)^{\gamma_0}.
\end{align*}
Substituting it into \eqref{var nable u-q gam} and using \eqref{nabla u-nabla w gamma}, we obtain
\begin{align*}
    &\Bigg(\fint_{B_{\varepsilon r}(x_0)}|\nabla u-{\bf q}_{x_0,\varepsilon r;\gamma_0}(u)|^{\gamma_0}\Bigg)^{\frac{1}{\gamma_0}}\\
    &\leq C\varepsilon^{\alpha}\Bigg(\fint_{B_r(x_0)}|\nabla u-{\bf q}_{x_0,r;\gamma_0}(u)|^{\gamma_0}\Bigg)^{\frac{1}{\gamma_0}}+C\varepsilon^{-\frac{n}{\gamma_0}}\Bigg(\fint_{B_{2r}(x_0)}|\nabla u-\nabla w|^{\gamma_0}\Bigg)^{\frac{1}{\gamma_0}}\\
    &\quad +C\varepsilon^{-\frac{n}{\gamma_0}}\omega(r)\Bigg(\fint_{B_{2r}(x_0)}\big(|\nabla u|+s\big)^{\gamma_0}\Bigg)^{\frac{1}{\gamma_0}}\\
    &\leq C\varepsilon^{\alpha} \Bigg(\fint_{B_r(x_0)}|\nabla u-{\bf q}_{x_0,r;\gamma_0}(u)|^{\gamma_0}\Bigg)^{\frac{1}{\gamma_0}}+C\varepsilon^{-\frac{n}{\gamma_0}}\Bigg(\frac{|\mu|(B_{2r}(x_0))}{r^{n-1}}\Bigg)^{\frac{1}{p-1}}\\
    &\quad+C\varepsilon^{-\frac{n}{\gamma_0}}\frac{|\mu|(B_{2r}(x_0))}{r^{n-1}}\fint_{B_{2r}(x_0)}\big(|\nabla u|+s\big)^{2-p}\ dx\\&\quad+C\varepsilon^{-\frac{n}{\gamma_0}}\omega(r)\Bigg(\fint_{B_{2r}(x_0)}\big(|\nabla u|+s\big)^{2-p}\ dx\Bigg)^{\frac{1}{2-p}}.
    \end{align*}
    Thus, \eqref{p<2 psi u var r} is proved.

{\bf Case 3: $1<p<3/2$.} By the proof of \eqref{p<2 psi u var r} and utilizing \eqref{nabla u-nabla w gamma02}, we derive \eqref{p<3/2 psi u var r}.
\end{proof}
\begin{proof}[Proof of Theorem \ref{main result Pwg}] We prove this theorem at the Lebesgue point $x=x_0$ of the vector-valued function $\nabla u$, assuming $B_R(x_0)\in \Omega$.

{\bf Case 1: $p\geq2$.} Choose $\varepsilon=\varepsilon(n, p, \lambda, \alpha)\in(0,1/4)$ sufficiently small such that $C\varepsilon^{\alpha}\leq 1/4$, where $C$ is the constant from \eqref{ite-phi-u}. For an integer $j\geq0$, let $r_j=\varepsilon^jR$, $B^j=B_{2r_j}(x_0)$, and
\begin{equation*}
    S_j=\fint_{B^j}\big(|\nabla u|+s\big)\ dx,\quad  {\bf q}_j={\bf q}_{x_0,r_j}, \quad\phi_j=\phi_u(x_0,r_j).
\end{equation*}
From \eqref{ite-phi-u}, we obtain
\begin{equation*}
    \phi_{j+1}\leq \frac{1}{4}\phi_j+C\Bigg(\frac{|\mu|(B^j)}{r_j^{n-1}}\Bigg)^{\frac{1}{p-1}}+C(\omega(r_j))^{\frac{2}{p}}S_j.
\end{equation*}
Let $j_0$ and $d$ be positive integers with $j_0\leq d$. Summing the above expression over $j=j_0, j_0+1,..., d$, we obtain
\begin{equation}\label{d+1 phi j}
    \sum_{j=j_0}^{d+1}\phi_j\leq C\phi_{j_0}+C\sum_{j=j_0}^d\Bigg(\frac{|\mu|(B^j)}{r_j^{n-1}}\Bigg)^{\frac{1}{p-1}}+C\sum_{j=j_0}^d(\omega(r_j))^{\frac{2}{p}}S_j.
\end{equation}
Since
\begin{equation*}
    |{\bf q}_{j+1}-{\bf q}_j|\leq |{\bf q}_{j+1}-\nabla u(x)|+|{\bf q}_j-\nabla u(x)|,
\end{equation*}
taking the average integral over $x\in B_{r_{j+1}}(x_0)$ yields
\begin{equation*}
    |{\bf q}_{j+1}-{\bf q}_j|\leq C\phi_j+C\phi_{j+1},
\end{equation*}
which implies
\begin{equation*}
    |{\bf q}_{d+1}-{\bf q}_{j_0}|\leq C\sum_{j=j_0}^{d+1}\phi_j.
\end{equation*}
This together with \eqref{d+1 phi j} gives
\begin{align}\label{q d+1 +sum phi}
    |{\bf q}_{d+1}|+\sum_{j=j_0}^{d+1}\phi_j\leq |{\bf q}_{j_0}|+C\phi_{j_0}+C\sum_{j=j_0}^d\Bigg(\frac{|\mu|(B^j)}{r_j^{n-1}}\Bigg)^{\frac{1}{p-1}}+C\sum_{j=j_0}^d(\omega(r_j))^{\frac{2}{p}}S_j.
\end{align}
Since
\begin{equation*}
    |{\bf q}_{j_0}|\leq |{\bf q}_{j_0}-\nabla u(x)|+|\nabla u(x)|,
\end{equation*}
taking the average integral over $x\in B_{r_{j_0}}(x_0)$ yields
\begin{equation}\label{q j0}
    |{\bf q}_{j_0}|\leq C\phi_{j_0}+C\fint_{B^{j_0}}|\nabla u(x)|\ dx\leq CS_{j_0}.
\end{equation}
Here, we used the definition of $\phi_{j_0}$, 
\begin{equation*}
    \phi_{j_0}\leq C\fint_{B^{j_0}}|\nabla u|\ dx\leq CS_{j_0}.
\end{equation*}
Furthermore, the comparison principle for Riemann integrals yields
\begin{equation}\label{sum Bj mu}
    \sum_{j=j_0}^d\Bigg(\frac{|\mu|(B^j)}{r_j^{n-1}}\Bigg)^{\frac{1}{p-1}}\leq C\int_0^{2r_{j_0-1}}\Bigg(\frac{|\mu|(B_t(x_0))}{t^{n-1}}\Bigg)^{\frac{1}{p-1}}\frac{dt}{t}.
\end{equation}
Using \eqref{q d+1 +sum phi}-\eqref{sum Bj mu}, we obtain
\begin{align}\label{q d+1 sum phi j}
    |{\bf q}_{d+1}|+\sum_{j=j_0}^{d+1}\phi_j\leq CS_{j_0}+C\int_0^{2r_{j_0-1}}\Bigg(\frac{|\mu|(B_t(x_0))}{t^{n-1}}\Bigg)^{\frac{1}{p-1}}\frac{dt}{t}+C\sum_{j=j_0}^d(\omega(r_j))^{\frac{2}{p}}S_j.
\end{align}
By \eqref{Dini-1} and the comparison principle for Riemann integration, there exists a sufficiently large $j_0=j_0(n, p, \varepsilon, C, \omega)>1$ such that
\begin{equation}\label{omega <1/10}
    (2\varepsilon)^{-n}C\sum_{j=j_0}^{\infty}(\omega(r_j))^{\frac{2}{p}}\leq \frac{1}{10},
\end{equation}
where $C$ is the constant from \eqref{q d+1 sum phi j}.

To prove \eqref{p>2 nabla u point} at $x=x_0$, it suffices to show that
\begin{equation}\label{nable u point 2}
    |\nabla u(x_0)|\leq CS_{j_0}+C\int_0^{2r_{j_0-1}}\Bigg(\frac{|\mu|(B_t(x_0))}{t^{n-1}}\Bigg)^{\frac{1}{p-1}}\frac{dt}{t}.
\end{equation}
We now consider the following cases to demonstrate the validity of \eqref{nable u point 2}.

{ Case i:} If $|\nabla u(x_0)|\leq S_{j_0}$, \eqref{nable u point 2} holds trivially.

{ Case ii:} If $S_j<|\nabla u(x_0)|$, $\forall j_0<j\leq j_1$, and we have $|\nabla u(x_0)|\leq S_{j_1+1}$. Based on the definitions of $\phi_{j_1}$ and ${\bf q}_{j_1}$, as well as the triangle inequality, we obtain
\begin{align}\label{case2 nabla u point}
    |\nabla u(x_0)|&\leq\fint_{B^{j_1+1}}\big(|\nabla u|+s\Big)\ dx\nonumber=\fint_{B^{j_1+1}}|\nabla u|\ dx+s\nonumber\\
    &\leq (2\varepsilon)^{-n}\fint_{B_{r_{j_1}}(x_0)}|\nabla u|\ dx+s
    \leq (2\varepsilon)^{-n}(\phi_{j_1}+|{\bf q}_{j_1}|)+s.
\end{align}
Applying \eqref{q d+1 sum phi j} with $d=j_1-1$, and using \eqref{case2 nabla u point}, we obtain

\begin{align*}
    |\nabla u(x_0)|&\leq (2\varepsilon)^{-n}CS_{j_0}+(2\varepsilon)^{-n}C\int_0^{2r_{j_0-1}}\Bigg(\frac{|\mu|(B_t(x_0))}{t^{n-1}}\Bigg)^{\frac{1}{p-1}}\frac{dt}{t}\\
   &\quad +(2\varepsilon)^{-n}C\sum_{j=j_0}^d(\omega(r_j))^{\frac{2}{p}}|\nabla u(x_0)|+s.
\end{align*}
 Hence, combining the above with \eqref{omega <1/10}, we obtain
\begin{equation*}
    |\nabla u(x_0)|\leq CS_{j_0}+C\int_0^{2r_{j_0-1}}\Bigg(\frac{|\mu|(B_t(x_0))}{t^{n-1}}\Bigg)^{\frac{1}{p-1}}\frac{dt}{t}+\frac{1}{10}|\nabla u(x_0)|+s.
\end{equation*}
This yields \eqref{nable u point 2}.

{ Case iii:} If $S_j<\nabla u(x_0)$ for any $j\leq j_0$, by \eqref{q d+1 sum phi j} and \eqref{omega <1/10}, we have for any $d>j_0$,
\begin{align}\label{Case3 q d+1}
    |{\bf q}_{d+1}|\leq CS_{j_0}+C\int_0^{2r_{j_0-1}}\Bigg(\frac{|\mu|(B_t(x_0))}{t^{n-1}}\Bigg)^{\frac{1}{p-1}}\frac{dt}{t}+\frac{1}{10}|\nabla u(x_0)|.
\end{align}
Since
\begin{equation*}
    |{\bf q}_{x,\rho}-\nabla u(x)|\leq |{\bf q}_{x,\rho}-\nabla u(z)|+|\nabla u(z)-\nabla u(x)|,
\end{equation*}
taking the average integral over $z\in B_{\rho}(x)$ yields
\begin{equation*}
    |{\bf q}_{x,\rho}-\nabla u(x)|\leq C\phi_u(x,\rho)+C\fint_{B_{\rho}(x)}|\nabla u(z)-\nabla u(x)|\ dz.
\end{equation*}
By the definition of $\phi_u$, we obtain for any Lebesgue point $x\in\Omega$ of $\nabla u$,
\begin{equation}\label{lim q =nabla u}
    \lim_{\rho\rightarrow0}{\bf q}_{x,\rho}=\nabla u(x).
\end{equation}
As $d\rightarrow\infty$, using \eqref{lim q =nabla u} and \eqref{Case3 q d+1}, we obtain
\begin{equation*}
    |\nabla u(x_0)|\leq CS_{j_0}+C\int_0^{2r_{j_0-1}}\Bigg(\frac{|\mu|(B_t(x_0))}{t^{n-1}}\Bigg)^{\frac{1}{p-1}}\frac{dt}{t}+\frac{1}{10}|\nabla u(x_0)|.
\end{equation*}
We arrive at \eqref{nable u point 2}. Thus, the proof of \eqref{p>2 nabla u point} is completed.

{\bf Case 2: $3/2\leq p<2$.}  
The proof of \cite[Theorem 1.1]{dz2024} is also applicable to  \eqref{3/2<p<2 nabla u point}. Thus, we omit the details.

{{\bf Case 3:} $ 1<p<3/2$.} Choose $\varepsilon=\varepsilon(n, p, \lambda, \alpha)\in(0,1/4)$ sufficiently small such that $C\varepsilon^{\alpha}\leq 1/4$, where $C$ is the constant from \eqref{p<3/2 psi u var r}. For an integer $j\geq0$, let $r_j=\varepsilon^jR$, $B^j=B_{2r_j}(x_0)$, and instead we define
\begin{equation*}
    U_j:=\Bigg(\fint_{B^j}\big(|\nabla u|+s\big)^{\gamma_0}\ dx\Bigg)^{\frac{1}{\gamma_0}},\quad {\bf q}_j={\bf q}_{x_0,r_j;\gamma_0},\quad \psi_j=\psi_u(x_0,r_j).
\end{equation*}
Then by replicating the proof of \eqref{p>2 nabla u point} and \eqref{3/2<p<2 nabla u point}, and using \eqref{p<3/2 psi u var r}, we obtain \eqref{1<p<3/2 nabla u point}. Therefore, Theorem \ref{main result Pwg} is proved.
\end{proof}

\section{Modulus of continuity estimates of gradient}\label{sec-modulus}
In this section, we derive an a modulus of continuity estimate of the gradient and complete the proof of Theorem \ref{main-result-cont}.

Denote 
\begin{align}\label{def-tildeW}
{\bf{\tilde W}}_{\frac{1}{p},p}^\rho(|\mu|)(x)=\sum_{i=1}^{\infty}\varepsilon^{\alpha_1 i}\big({\bf W}_{\frac{1}{p},p}^{\varepsilon^{-i}\rho}(|\mu|)(x)[\varepsilon^{-i}\rho\leq R/2]+{\bf W}_{\frac{1}{p},p}^{R/2}(|\mu|)(x)[\varepsilon^{-i}\rho>R/2]\big)
\end{align}
and
\begin{align}\label{def tilde I}
    {\bf {\tilde I}}_1^{\rho}(|\mu|)(x)=\sum_{i=1}^{\infty}\varepsilon^{\alpha_1i}\big({\bf I}_1^{\varepsilon^{-i}\rho}(|\mu|)(x)[\varepsilon ^{-i}\rho\leq R/2]+{\bf I}_1^{R/2}(|\mu|)(x)[\varepsilon ^{-i}\rho>R/2]\big).
\end{align}
Set 
\begin{equation*}
h(x,r)=\big(\frac{|\mu|(B_{r}(x))}{r^{n-1}}\big)^{\frac{1}{p-1}}, \quad g(x,r)=h(x,r)^{p-1}
\end{equation*}
and define 
\begin{equation*}
\tilde h(x,t):=\sum_{i=1}^{\infty}\varepsilon^{\alpha_1 i}\big(h(x,\varepsilon^{-i}t)[\varepsilon^{-i}t\leq R/2]+h(x,R/2)[\varepsilon^{-i}t>R/2]\big)
\end{equation*}
and
\begin{equation}\label{def-tildeh}
    {\tilde g}(x,t):=\sum_{i=1}^{\infty}\varepsilon^{\alpha_1i}\big(g(x,\varepsilon^{-i}t)[\varepsilon^{-i}t\leq R/2]+g(x,R/2)[\varepsilon^{-i}t>R/2]\big).
\end{equation}
Here we used the Iverson bracket notation, i.e., $[P]=1$ if $P$ is true and $[P]=0$ otherwise. By Proposition \ref{prop-phi-Du} and using a similar argument as in the proof of Lemma \ref{lem-phi-w}, we derive the following result.

\begin{lemma}
Let $u\in W_{\text{loc}}^{1,p}(\Omega)$ be a solution to \eqref{p-laplace}, $B_{2r}(x)\subset\subset B_{R}(x_0)\subset\Omega$ with $r\leq R/4$, and let $\alpha_1\in (0,\alpha)$ be given in \eqref{alpha1}. Then for any $\rho\in(0,r]$, the following assertions hold:

(i) if $p\geq 2$, 
\begin{align}\label{iteration-u}
\phi_u(x,\rho)\leq C\big(\frac{\rho}{r}\big)^{\alpha_1}\phi_u(x,r)+C\tilde h(x,2\rho)+C\tilde\omega(\rho)\big(\|\nabla u\|_{L^{\infty}(B_{2r}(x))}+s\big),
\end{align}
and 
\begin{align}\label{iteration-sum-u}
\sum_{j=0}^{\infty}\phi_u(x,\varepsilon^j\rho)\leq C\big(\frac{\rho}{r}\big)^{\alpha_1}\phi_u(x,r)+C\int_0^{\rho}\frac{\tilde h (x,t)}{t}\ dt+C\big(\|\nabla u\|_{L^{\infty}(B_{2r}(x))}+s\big)\int_{0}^{\rho}\frac{\tilde\omega(t)}{t}\ dt,
\end{align}
where $C$ is a constant depending on $n$, $p$, and $\lambda$, and $\tilde\omega(\cdot)$ is defined in \eqref{def-tilde-omega};

(ii) if $1<p<2$,
\begin{align*}
    \psi_u(x,\rho)&\leq C\big(\frac{\rho}{r}\big)^{\alpha_1}\psi_u(x,r)+C {\tilde h}(x,2\rho)+C{\tilde g}(x,2\rho)\big(\|\nabla u\|_{L^{\infty}(B_{2r}(x))}+s\big)^{2-p}\\
   &\quad +C{\overline{\omega}}(\rho)\big(\|\nabla u\|_{L^{\infty}(B_{2r}(x))}+s\big),
\end{align*}
and 
\begin{align*}
    \sum_{j=0}^{\infty}\psi_u(x,\varepsilon^j\rho)&\leq C\big(\frac{\rho}{r}\big)^{\alpha_1}\psi_u(x,r)+C\int_0^{\rho}\frac{{\tilde h}(x,t)}{t}\ dt\nonumber\\
    &\quad +C\big(\|\nabla u\|_{L^{\infty}(B_{2r}(x))}+s\big)^{2-p}\int_0^{\rho}\frac{{\tilde g}(x,t)}{t}\ dt\nonumber\\
    &\quad+C\big(\|\nabla u\|_{L^{\infty}(B_{2r}(x))}+s\big)\int_0^{\rho}\frac{{\overline{w}}(t)}{t}\ dt,
\end{align*}
where $C$ is a constant depending on $n$, $p$, $\lambda$, and $\gamma_0$, and $\overline{\omega}(\cdot)$ is defined in \eqref{def overline omega}.
\end{lemma}
We now present the continuity estimate for the gradient.
\begin{theorem}\label{thm continuity}
Under the conditions of Theorem \ref{main-result-cont} and assume $\alpha_1\in(0,\alpha)$ with $\alpha$ given in Theorem \ref{thm v-BMO}. Then for any $R\in(0,1]$, $B_R(x_0)\subset\Omega$, and for any Lebesgue points $x,y\in B_{R/4}(x_0)$ of $\nabla u$, there exists a constant $C=C(n, p, \lambda, \alpha_1, \omega)$, it holds that,

(i) when $p\geq2$,
\begin{align}\label{p>2 ux-uy}
|\nabla u(x)-\nabla u(y)|\leq C\mathcal{M}_1\left(\big(\frac{\rho}{R}\big)^{\alpha_1}+\int_{0}^{\rho}\frac{\tilde\omega(t)}{t}\ dt\right)+C\|{\bf{\tilde W}}_{\frac{1}{p},p}^\rho(|\mu|)\|_{L^\infty(B_{R/4}(x_0))},
\end{align}
where $\rho=|x-y|$, $\tilde\omega$ and ${\bf{\tilde W}}_{\frac{1}{p},p}^\rho$ are defined in \eqref{def-tilde-omega} and \eqref{def-tildeW}, respectively, and 
\begin{equation*}
\mathcal{M}_1:=\|{\bf W}^R_{\frac{1}{p},p}(|\mu|)\|_{L^\infty(B_R(x_0))}+CR^{-n}\||\nabla u|+s\|_{L^1(B_{R}(x_0))};
\end{equation*}

(ii) when $3/2\leq p<2$,
\begin{align}\label{p<2 nabla ux-nabla uy}
    |\nabla u(x)-\nabla u(y)|&\leq C\mathcal{M}_2\Bigg(\big(\frac{\rho}{R}\big)^{\alpha_1}+\int_0^{\rho}\frac{\overline{\omega}(t)}{t}\ dt\Bigg)+C\|{\bf {\tilde W}}_{\frac{1}{p},p}^{\rho}(|\mu|)\|_{L^{\infty}(B_{R/4}(x_0))}\nonumber\\
    &\quad +C\mathcal{M}_2^{2-p}\|{\bf {\tilde I}}_1^{\rho}(|\mu|)\|_{L^{\infty}(B_{R/4}(x_0))},
\end{align}
where $\rho=|x-y|$, $\overline{\omega}$, ${\bf {\tilde W}}_{\frac{1}{p},p}^{\rho}$, and ${\bf {\tilde I}}_1^{\rho}$ are defined in \eqref{def-tildeW} and \eqref{def tilde I}, respectively, and 
\begin{equation*}
    \mathcal{M}_2:=\|{\bf I}_1^R(|\mu|)\|_{L^{\infty}(B_R(x_0))}^{\frac{1}{p-1}}+R^{-\frac{n}{2-p}}\||\nabla u|+s\|_{L^{2-p}(B_R(x_0))}.
\end{equation*}

(iii) when $1<p<3/2$,
\begin{align}\label{p<3/2 nabla ux-nabla uy}
    |\nabla u(x)-\nabla u(y)|&\leq C\mathcal{M}_3\Bigg(\big(\frac{\rho}{R}\big)^{\alpha_1}+\int_0^{\rho}\frac{\overline{\omega}(t)}{t}\ dt\Bigg)+C\|{\bf {\tilde W}}_{\frac{1}{p},p}^{\rho}(|\mu|)\|_{L^{\infty}(B_{R/4}(x_0))}\nonumber\\
    &\quad +C\mathcal{M}_3^{2-p}\|{\bf {\tilde I}}_1^{\rho}(|\mu|)\|_{L^{\infty}(B_{R/4}(x_0))},
\end{align}
where $\rho=|x-y|$, $\overline{\omega}$, ${\bf {\tilde W}}_{\frac{1}{p},p}^{\rho}$, and ${\bf {\tilde I}}_1^{\rho}$ are defined in \eqref{def-tildeW} and \eqref{def tilde I}, respectively, and 
\begin{equation*}
    \mathcal{M}_3:=\|{\bf I}_1^R(|\mu|)\|_{L^{\infty}(B_R(x_0))}^{\frac{1}{p-1}}+R^{-\frac{2n}{(p-1)^2}}\|\nabla u|+s\|_{L^{(p-1)^2/2}(B_R(x_0))}.
\end{equation*}
\end{theorem}

\begin{proof}
When $p\geq 2$, for any Lebesgue points $x$, $y\in B_{R/4}(x_0)$ of $\nabla u$, we let $\rho=|x-y|$. If $\rho \geq R/8$, we can readily conclude that
\begin{equation*}
    |\nabla u(x)-u(y)|\leq 2\|\nabla u\|_{L^{\infty}(B_{R/2}(x_0))}\leq C\big(\frac{\rho}{R}\big)^{\alpha_1}\|\nabla u\|_{L^{\infty}(B_{R/2}(x_0))}.
\end{equation*}
If $\rho<R/8$, similar to the proof of \eqref{nabla-q w}, we have
\begin{equation}\label{sum phi u}
    |\nabla u(x)-{\bf q}_{x,\rho}|\leq C\sum\limits^{\infty}_{j=0}\phi_u(x,\varepsilon^j\rho),
\end{equation}
where ${\bf q}_{x,\rho}:={\bf q}_{x,\rho}(u)$.
Utilizing the triangle inequality, for any Lebesgue points $x$, $y\in B_{R/4}(x_0)$ of $\nabla u$, we obtain 
\begin{align*}
    |\nabla u(x)-\nabla u(y)|&\leq |\nabla u(x)-{\bf q}_{x,\rho}|+|{\bf q}_{x,\rho}-{\bf q}_{y,\rho}|+|\nabla u(y)-{\bf q}_{y,\rho}|\\
   & \leq 2 \sup_{y_0\in B_{R/4}(x_0)}|\nabla u(y_0)-{\bf q}_{y_0,\rho}|+|\nabla u(z)-{\bf q}_{x,\rho}|+|\nabla u(z)-{\bf q}_{y,\rho}|.
\end{align*}
By taking average integral over $z\in B(x,\rho)\cap B(y,\rho)$, and applying \eqref{sum phi u}, we get 
\begin{align}\label{ux-uy}
    |\nabla u(x)-\nabla u(y)|&\leq C \sup_{y_0\in B_{R/4}(x_0)}|\nabla u(y_0)-{\bf q}_{y_0, \rho}|+C\phi_u(x,\rho)+C\phi_u(y,\rho)\nonumber\\
    &\leq C\sup_{y_0\in B_{R/4}(x_0)}\sum\limits_{j=0}^{\infty}\phi_u(y_0,\varepsilon^j\rho).
\end{align}
Using \eqref{iteration-sum-u} and \eqref{ux-uy}  replacing $r$ with $R/8$, and $B_{R/4}(y_0)\subset B_{R/2}(x_0)$, $\forall y_0\in B_{R/4}(x_0)$, we obtain 
\begin{align}\label{ux-uy2}
    |\nabla u(x)-\nabla u(y)|&\leq C\big(\frac{\rho}{R}\big)^{\alpha_1}\|\nabla u\|_{L^{\infty}(B_{R/2}(y_0))}+C\sup_{y_0\in B_{R/4}(y_0)}\int_0^{\rho}\frac{\tilde h(y_0,t)}{t}\ dt\nonumber\\
    &\quad +C(\|\nabla u\|_{L^{\infty}(B_{R/2}(y_0))}+s)\int^{\rho}_0\frac{\tilde \omega(t)}{t}dt.
\end{align}
 For any $y_0\in B_{R/4}$ and $\rho\in(0,R/2)$, by the definition of $\tilde h(x,t)$ in \eqref{def-tildeh} and \cite[Lemma 4.2]{dz2024}, as well as the comparison principle for Riemann integrals, we derive 
\begin{align*}
   \int_0^{\rho}\frac{\tilde h(y_0,t)}{t} \ dt=&\sum_{i=1}^{\infty}\varepsilon^{\alpha_1i}\int_0^{\rho}\frac{h(y_0,\varepsilon^{-i}t)}{t}[\varepsilon^{-i}t\leq R/2]\ dt\\&+\sum_{i=1}^{\infty}\varepsilon^{\alpha_1i}\int_0^{\rho}\frac{h(y_0,R/2)}{t}[\varepsilon^{-i}t>R/2]\ dt.
\end{align*}
The first term on the right-hand side is equal to
\begin{align*}
&\sum_{i=1}^{\infty}\varepsilon^{\alpha_1i}\Bigg(\int_0^{\rho}\frac{h(y_0,\varepsilon^{-i}t)}{t}\ dt[\varepsilon^{-i}\rho\leq R/2]+\int_0^{\varepsilon^iR/2}\frac{h(y_0,\varepsilon^{-i}t)}{t} \ dt[\varepsilon^{-i}\rho>R/2]\Bigg)\\
&=\sum_{i=1}^{\infty}\varepsilon^{\alpha_1i}\Bigg(\int_0^{\varepsilon^{-i}\rho}\frac{h(y_0,t)}{t}\ dt[\varepsilon^{-i}\rho\leq R/2]+\int^{R/2}_0\frac{h(y_0,t)}{t}\ dt[\varepsilon^{-i}\rho>R/2]\Bigg)\\
&=\sum_{i=1}^{\infty}\varepsilon^{\alpha_1i}\big({\bf W}_{\frac{1}{p},p}^{\varepsilon^{-i}\rho}(|\mu|(y_0)[\varepsilon^{-i}\rho\leq R/2]+{\bf W}_{\frac{1}{p},p}^{R/2}(|\mu|(y_0)[\varepsilon^{-i}\rho>R/2]\big)\\
&={\bf \tilde W}_{\frac{1}{p},p}^{\rho}(|\mu|)(y_0).
\end{align*}
The second term is equal to
\begin{equation*}
\sum_{i=1}^{\infty}\varepsilon^{\alpha_1i}\int_{\varepsilon^iR/2}^{\rho}\frac{h(y_0,R/2)}{t}\ dt[\varepsilon^{-i}\rho>R/2]=\sum_{i=1}^{\infty}\varepsilon^{\alpha_1i}\ln(2\varepsilon^{-i}\rho/R)[\varepsilon^{-i}\rho>R/2]h(y_0,R/2).
\end{equation*}
Now we let $M>0$ be the integer such that $\varepsilon^{-M+1}\rho\leq R/2<\varepsilon^{-M}\rho$. Then we have
\begin{align*}
&h(y_0,R/2)\sum\limits_{i=1}^{\infty}\varepsilon^{\alpha_1i}\ln(2\varepsilon^{-i}\rho/R)[\varepsilon^{-i}\rho>R/2]\\
&=h(y_0,R/2)\sum\limits_{i=M}^{\infty}\varepsilon^{\alpha_1i}\ln(2\varepsilon^{-i}\rho/R)\\
&=\varepsilon^{\alpha_1M}h(y_0,R/2)\sum\limits_{i=M}^{\infty}\varepsilon^{\alpha_1(i-M)}\big(\ln(2\varepsilon^{-M}\rho/R)+(i-M)\ln(\varepsilon^{-1})\big)\\
&\leq\big(\frac{2\rho}{R}\big)^{\alpha_1}h(y_0,R/2)\sum\limits_{i=M}^{\infty}\varepsilon^{\alpha_1(i-M)}(i-M+1)\ln(\varepsilon^{-1})\leq C\big(\frac{\rho}{R}\big)^{\alpha_1}{\bf W}_{\frac{1}{p},p}^R(|\mu|)(y_0).
\end{align*}
Accordingly, we have
\begin{align}\label{sum tilde h}
    \int_0^{\rho}\frac{\tilde h(y_0,t)}{t}\ dt\leq{\bf {\tilde W}}_{\frac{1}{p},p}^{\rho}(|\mu|)(y_0)+C\big(\frac{\rho}{R}\big)^{\alpha_1}{\bf W}^R_{\frac{1}{p},p}(|\mu|)(y_0).
\end{align}
It follows from \eqref{p>2 nabla u point} that
\begin{equation*}
    \|\nabla u\|_{L^{\infty}(B_{R/2}(y_0))}\leq C\|{\bf W}_{\frac{1}{p},p}^R(|\mu|)\|_{L^{\infty}(B_R(y_0))}+CR^{-n}\||\nabla u|+s\|_{L^1(B_R(x))}.
\end{equation*}
Substituting the above and \eqref{sum tilde h} into \eqref{ux-uy2} yields \eqref{p>2 ux-uy}.

When $1<p<2$, by following a similar argument as in the proof of \eqref{p>2 ux-uy} and using the fact that
\begin{equation}\label{ine-1p2}
{\bf W}^\rho_{\frac{1}{p},p}(|\mu|)(x)\leq C\big({\bf I}_1^{2\rho}(|\mu|)(x)\big)^{\frac{1}{p-1}}\quad\mbox{when}~1<p<2,
\end{equation}
we obtain \eqref{p<2 nabla ux-nabla uy} and \eqref{p<3/2 nabla ux-nabla uy}. The details are omitted here.
\end{proof}

\begin{proof}[Proof of Theorem \ref{main-result-cont}.]
It follows from \eqref{def-tildeW} and \eqref{def tilde I} that
\begin{align*}
    \|{\bf {\tilde W}}_{\frac{1}{p},p}^{\rho}(|\mu|)\|_{L^{\infty}(B_{R/4}(x_0))}&\leq \sum_{i=1}^{\infty}\varepsilon^{\alpha_1i}\big(\|{\bf W}_{\frac{1}{p},p}^{\varepsilon^{-i}\rho}(|\mu|)\|_{L^{\infty}(B_{R/4}(x_0))}[\varepsilon^{-i}\rho\leq R/2]\\
    &\quad +\|{\bf W}_{\frac{1}{p},p}^{R/2}(|\mu|)\|_{L^{\infty}(B_{R/4}(x_0))}[\varepsilon^{-i}\rho>R/2]\big)
\end{align*}
and 
\begin{align*}
    \|{\bf {\tilde I}}_{1}^{\rho}(|\mu|)\|_{L^{\infty}(B_{R/4}(x_0))}&\leq \sum_{i=1}^{\infty}\varepsilon^{\alpha_1i}\big(\|{\bf I}_{1}^{\varepsilon^{-i}\rho}(|\mu|)\|_{L^{\infty}(B_{R/4}(x_0))}[\varepsilon^{-i}\rho\leq R/2]\\
    &\quad +\|{\bf I}_{1}^{R/2}(|\mu|)\|_{L^{\infty}(B_{R/4}(x_0))}[\varepsilon^{-i}\rho>R/2]\big).
\end{align*}
By \eqref{assump-Wp}, \eqref{assump Ip}, and the dominated convergence theorem, $\|{\bf {\tilde W}}_{\frac{1}{p},p}^{\rho}\|_{L^{\infty}(B_{R/4}(x_0))}$ and $\|{\bf {\tilde I}}_{1}^{\rho}(|\mu|)\|_{L^{\infty}(B_{R/4}(x_0))}$ converge to $0$ as $\rho\rightarrow 0$. Then combining the fact that the set of Lebesgue points of $\nabla u$ is dense in $\Omega$ and Theorem \ref{thm continuity}, we derive the continuity of $\nabla u$ and thus finish the proof of Theorem \ref{main-result-cont}.
\end{proof}

\section{Global pointwise estimates of the gradient}\label{sec-boundary}
In this section, we primarily aim to establish a global pointwise gradient estimate for the following problem
\begin{equation}\label{bou u=0}
    \begin{cases}
    -\text{div}(a(x)(|\nabla u|^2+s^2)^{\frac{p-2}{2}}\nabla u)=\mu\quad &in \quad \Omega,\\
    u=0\quad &on \quad \partial \Omega,
    \end{cases}
\end{equation}
where $a(\cdot)$ satisfies Assumption \ref{assumption 2}, and $\Omega$ has a $C^{1,\text{DMO}}$ boundary characterized by $R_0$ and $\varrho_0$ as in Definition \ref{def chi}.

\subsection{Auxillary results}
For any fixed point $x_0\in\partial\Omega$, we will derive a gradient estimate near it. Without loss of generality, we assume $x_0=0$. Then we select a local coordinate system at $x_0=0$ and a $C^{1,\text{DMO}}$ function $\chi$ as in Definition $\ref{def chi}$ such that $\chi(0')=0$. 

By using $|\nabla_{x'}\chi(0')|=0$ and \eqref{sup chi-chi}, we get that there exists $R'=R'(\varrho_0,R_0)\in(0,R_0)$ such that
\begin{equation}\label{chi <1/2}
    |\nabla_{x'}\chi(x')|\leq \frac{1}{2}\quad if\quad |x'|\leq R'.
\end{equation}
Denote
\begin{equation*}
    \Gamma(y)=(y'+\chi(y'),y_n) \quad and\quad \Lambda (x)=\Gamma^{-1}(x)=(x',x_n-\chi(x')). 
\end{equation*}
Then it follows from  \cite[(2.8)]{MR2019} that
\begin{equation}\label{Omeag in Gamma in  Omeag}
    \Omega_{r/2}\subset\Gamma(B_r^+)\subset\Omega_{2r}\quad \forall r\in(0,R'/2],
\end{equation}
where $\Omega_r=\Omega\cap B_r$.
Therefore, there exist constants $c_1(n)$ and $c_2(n)$ depending only on $n$ such that for any $x\in\bar{\Omega}$ and $0<r<R'$, it holds
\begin{equation}\label{c1<Omega<c2}
    c_1(n)r^n\leq|\Omega_r(x)|\leq c_2(n)r^n.
\end{equation}
Define $u_1(y)=u(\Gamma(y))$, $a_1(y)=a(\Gamma(y))$, and $\mu_1(B)=\mu(\Gamma(B))$  for any Borel set $B\subset\mathbb R^n$. Then from \eqref{bou u=0}, it follows that
\begin{equation}\label{equ u_1}
    \begin{cases}
        -\text{div}_y\big(a_1(y)(|D\Lambda \nabla_yu_1|^2+s^2)^{\frac{p-2}{2}}(D\Lambda)^TD\Lambda \nabla_yu_1\big)=\mu_1\quad &in\quad B^+_{R'/2},\\
        u_1=0\quad &on\quad B_{R'/2}\cap \partial\mathbb R^n_+.
    \end{cases}
\end{equation}
From \eqref{condition}, we have
\begin{equation}\label{a1 lambda}
   \lambda^{-1}|\eta|^2\leq a_{1,ij}(y)\eta_i\eta_j,\quad |a_{1,ij}(y)|\leq \lambda, \quad y\in B^+_{R'/2},\quad \forall\eta=(\eta_1,...,\eta_n)\in\mathbb R^n.
\end{equation}
We denote
\begin{equation*}
    A_1(y,\xi):=a_1(y)\big(|D\Lambda\xi|^2+s^2\big)^{\frac{p-2}{2}}(D\Lambda)^TD\Lambda\xi.
\end{equation*}
From \eqref{a1 lambda}, we have
\begin{equation}\label{A1 leq lambda1}
    |A_1(y,\xi)|\leq \lambda_1(n,p,\lambda)\big(|\xi|^2+s^2\big)^{\frac{p-1}{2}}.
\end{equation}
We define
\begin{equation*}
    \varrho_1(r):=\sup_{y\in B^+_{R'/2}}\Bigg(\fint_{B^+_{R'/2}\cap B_r(y)}|a_1(y)-(a_1)_{B^+_{R'/2}\cap B_r(y)}|^2\ dy\Bigg)^{1/2}.
\end{equation*}
By $a_1(y)=a(\Gamma(y))$ and \eqref{sup chi-chi}, we have $\varrho_1=\varrho+\varrho_0$. Then we have 
\begin{equation*}
    \int_0^1\frac{(\varrho_1(r))^\frac{2}{p}}{r}\ dr<+\infty\quad \mbox{if}\quad p\geq2,
\end{equation*}
and
\begin{equation*}
    \int_0^1\frac{\varrho_1(r)}{r}\ dr<+\infty\quad \mbox{if}\quad 1<p<2.
\end{equation*}

Let $4r\leq R'$, and consider the unique solution $w\in u_1+W_0^{1,p}(B_{2r}^+)$ to
\begin{equation}\label{w in u_1+W}
    \begin{cases}
        -\text{div}_y\big(a_1(y)(|D\Lambda \nabla_yw|^2+s^2\big)^{\frac{p-2}{2}}(D\Lambda)^TD\Lambda \nabla_yw\big)=0\quad &in\quad B_{2r}^+,\\
        w=u_1\quad & on \quad\partial B_{2r}^+.
    \end{cases}
\end{equation}

We first prove a boundary version of the reverse Hölder inequality.
\begin{lemma}
    Let $w$ be a solution to \eqref{w in u_1+W}. There exists a constant $\theta_1>p$ depending only on $n$, $p$, and $\lambda$, such that for any $t>0$ and for all $B_{\rho}^+(y_0)\subset B_{2r}^+$, the following estimate holds:
    \begin{equation}\label{Caccioppoli inequality}
        \Bigg(\fint_{B_{\rho/2}^+(y_0)}\big(|\nabla_yw|+s\big)^{\theta_1}\ dy\Bigg)^{1/\theta_1}\leq C\Bigg(\fint_{B_{\rho}^+(y_0)}\big(|\nabla_yw|+s\big)^t\ dy\bigg)^{1/t},
    \end{equation}
    where $C>0$ depending on $n$, $p$, $\lambda$, and $t$.
\end{lemma}
\begin{proof}
We shall prove the case of $p\geq 2$ as an example since $1<p<2$ follows from a similar argument and \cite[Lemma 5.1]{dz2024}.

    For simplicity, denote $\nabla$ as $\nabla_y$ in the proof. First, we establish a Caccioppoli-type inequality on a hemisphere. Let $y_0\in B_{2r}\cap\partial\mathbb R_+^n$ and $B_{2\rho}(y_0)\subset\subset B_{2r}$. Suppose $\zeta$ is a nonnegative smooth function satisfying $\zeta=1$ in $B_{\rho}(y_0)$, $|\nabla \zeta|\leq 2\rho^{-1}$, and $\zeta=0$ outside $B_{2\rho}(y_0)$. Taking $\zeta^pw$ as a test function for \eqref{w in u_1+W}, we obtain
    \begin{align}\label{A+B}
        0=&\int_{B^+_{2r}}\langle a_1(y)(|D\Lambda \nabla w|^2+s^2)^{\frac{p-2}{2}}(D\Lambda)^T D\Lambda \nabla w, \zeta^p\nabla w\rangle\ dy\nonumber\\
        &\quad+p\int_{B_{2r}^+}\langle a_1(y)(|D\Lambda \nabla w|^2+s^2)^{\frac{p-2}{2}}(D\Lambda)^TD\Lambda \nabla w, \zeta^{p-1}w\nabla \zeta\rangle\  dy =: A+B
    \end{align}
It follows from \eqref{a1 lambda} that
    \begin{align*}
        A&=\int_{B_{2r}^+}\zeta^p a_1(y)(D\Lambda\nabla w|^2+s^2)^{\frac{p-2}{2}}(D\Lambda)^TD\Lambda|\nabla w|^2\ dy\\
        &\geq c(n, p, \lambda)\int_{B_{2r}^+}\zeta^p|\nabla w|^p\ dy.
    \end{align*}
    Utilizing \eqref{A1 leq lambda1} and Young's inequality with exponents $p$ and $p/(p-1)$, we derive
    \begin{align*}
        |B|&\leq p\lambda_1\int_{B_{2\rho}^+(y_0)}(|\nabla w|^2+s^2)^{\frac{p-1}{2}}\zeta^{p-1}|w\nabla \zeta|\ dy\\
        &\leq \frac{c(n,p,\lambda)}{2} \int_{B_{2r}^+}\zeta^p(|\nabla w|^2+s^2)^{\frac{p}{2}}\ dy+C\int_{B_{2\rho}^+(y_0)}|w\nabla\zeta|^p\ dy.
    \end{align*}
    Thus, from \eqref{A+B}, we obtain 
    \begin{equation*}
        \int_{B_{\rho}^+(y_0)}|\nabla w|^p\ dy \leq C\rho^{-p}\int_{B_{2\rho}^+(y_0)}|w|^p\ dy+C\rho^ns^p.
    \end{equation*}
    Recalling that $w=u_1=0$ on $B_{2r}\cap\partial\mathbb R_+^n$, by using Sobolev embedding theorem and Poincaré inequality, for any $q$ satisfying max$\{1,\frac{np}{n+p}\}\leq q<p$, we get
    \begin{equation*}
        \Bigg(\int_{B_{2\rho}^+(y_0)}|w|^p\ dy\Bigg)^{1/p}\leq C\rho^{1+\frac{n}{p}-\frac{n}{q}}\Bigg(\int_{B_{2\rho}^+(y_0)}|\nabla w|^q\ dy\Bigg)^{1/q}.
    \end{equation*}
    Combining the last two inequalities yields
    \begin{equation*}
        \Bigg(\fint_{B_{\rho}^+(y_0)}\big(|\nabla w|+s\big)^p\ dy\Bigg)^{1/p}\leq C\Bigg(\fint_{B_{2\rho}^+(y_0)}\big(|\nabla w|+s\big)^q\ dy\Bigg)^{1/q}.
    \end{equation*}
    In a similar manner, we also obtain the interior estimate 
    \begin{equation*}
        \Bigg(\fint_{B_{\rho}(y_0)}\big(|\nabla w|+s\big)^p\ dy \Bigg)^{1/p}\leq C\Bigg(\fint_{B_{2\rho}(y_0)}\big(|\nabla w|+s\big)^q\ dy\Bigg)^{1/q},\quad \forall \ B_{2\rho}(y_0)\subset\subset B_{2r}^+.
    \end{equation*}
    Then using a standard covering argument and Gehring's lemma (see, for instance \cite[Chapter 6]{MR2003}), we derive \eqref{Caccioppoli inequality}.
\end{proof}

By a nearly identical proof, we also obtain a boundary estimate analogous to Lemma \ref{lem-comparison}.
\begin{lemma}
    Let $w$ be the solution to \eqref{w in u_1+W}, we obtain the following results:
    
    (i) when $p\geq 2$, there exists a constant $C=C(n, p, \lambda)>1$ such that
    \begin{equation}\label{bou nabla u_1-nabla w}
        \fint_{B_{2r}^+}|\nabla_yu_1-\nabla_yw|\ dy\leq C\Bigg(\frac{|\mu_1|(B_{2r}^+)}{r^{n-1}}\Bigg)^{\frac{1}{p-1}};
    \end{equation}

    (ii) when $3/2\leq p<2$, take $\gamma_0= 2-p$, there exists a constant $C=C(n, p, \lambda)>1$ such that
    \begin{align}\label{bou nabla u1-nabla w gamma p>3/2}
       & \Bigg(\fint_{B_{2r}^+}|\nabla_yu_1-\nabla_yw|^{\gamma_0}\ dy\Bigg)^{1/\gamma_0}\nonumber\\&\leq C\Bigg(\frac{|\mu_1|(B_{2r}^+)}{r^{n-1}}\Bigg)^{\frac{1}{p-1}}
        +C\frac{|\mu_1|(B_{2r}^+)}{r^{n-1}}\fint_{B_{2r}^+}\big(|\nabla u_1|+s\big)^{2-p}\ dy;
    \end{align}
    
    (iii) when $1<p<3/2$, take $\gamma_0=(p-1)^2/2$, there exists a constant $C=C(n, p, \lambda)>1$ such that
    \begin{align}\label{bou nabla u1-nab w gamma p<3/2}
        &\Bigg(\fint_{B_{2r}^+}|\nabla_yu_1-\nabla_yw|^{\gamma_0}\ dy\Bigg)^{1/\gamma_0}\nonumber\\
        &\leq C\Bigg(\frac{|\mu_1|(B_{2r}^+)}{r^{n-1}}\Bigg)^{\frac{1}{p-1}}+C\frac{|\mu_1|(B_{2r}^+)}{r^{n-1}}\Bigg(\fint_{B_{2r}^+}\big(|\nabla_yu_1|+s\big)^{\gamma_0}\ dy\Bigg)^{2-p/\gamma_0}.
    \end{align}
\end{lemma}
Analogous to the proof of Lemma \ref{lem-Dv-Dw}, we  obtain an estimate for the difference $\nabla v-\nabla w$.
\begin{lemma}\label{bou |nabla v-nabla w}
    Let $v\in w+W_0^{1,p}(B_r^+)$ be the unique solution to
    \begin{equation}\label{v=w on B_r+}
        \begin{cases}
            -\text{div}\big((a_1)_{B^+_r}(|\nabla_yv|^2+s^2\big)^{\frac{p-2}{2}}\nabla_yv=0\quad &in\quad B_r^+,\\
            v=w\quad &on \quad \partial B_r^+,
        \end{cases}
    \end{equation}
    where $w$ satisfies \eqref{w in u_1+W}. Then we have
    \begin{equation}\label{boundary nabla v-nabla w}
        \fint_{B_r^+}|\nabla v-\nabla w|\ dx\leq C(\varrho_1(r))^{\frac{2}{p}}\big(\|\nabla w\|_{L^{\infty}(B_r^+)}+s\big)\quad if \quad p\geq2,
    \end{equation}
    and 
    \begin{equation}\label{boundary |nabla v-nabla w| gamma}
        \fint_{B_r^+}|\nabla v-\nabla w|^{\gamma_0}\leq C(\varrho_1(r))^{\gamma_0}\big(\|\nabla w\|_{L^{\infty}(B_r^+)}+s\big)^{\gamma_0}\quad if \quad1<p<2, 0<\gamma_0<1.
    \end{equation}
\end{lemma}

By using Theorem \ref{thm v-BMO} and the arguments in the proof of \cite[Lemma 5.3]{dz2024}, we derive the oscillation estimates as follows.
\begin{lemma}\label{boundary v-BMO}
    Let $v$ be the unique solution to \eqref{v=w on B_r+}. The following estimates hold:
    
    (i) when $p\geq2$, then there exists a constant $C=C(n,p,\lambda)>1$ such that for any half ball $B_{\rho}^+\subset B_{ \rho'}^+\subset B_r^+$, we have
    \begin{equation*}
        \inf_{\theta\in\mathbb R}\Bigg(\fint_{B_{\rho}^+}|\nabla_{y_n}v-\theta|+\nabla_{y'}v|\Bigg)\leq C\big(\frac{\rho}{\rho'}\big)^{\alpha}\inf_{\theta\in\mathbb R}\Bigg(\fint_{B_{\rho'}^+}|\nabla_{y_n}v-\theta|+|\nabla_{y'}v|\Bigg);
    \end{equation*}

    (ii) when $1<p<2$, for any $0<\gamma_0<1$, there exists a constant $C=C(n,p,\lambda,\gamma_0)>1$ such that for any half ball $B_{\rho}^+\subset B_{\rho'}^+\subset B_r^+$, we have
    \begin{equation*}
        \inf_{\theta\in\mathbb R}\Bigg(\fint_{\rho^+}|\nabla_{y_n}v-\theta|^{\gamma_0}+|\nabla_{y'}v|^{\gamma_0}\Bigg)^{1/\gamma_0}\leq C\big(\frac{\rho}{\rho'})^{\alpha}\inf_{\theta\in\mathbb R}\Bigg(\fint_{B_{\rho'}^+}|\nabla_{y_n}v-\theta|^{\gamma_0}+|\nabla_{y'}v|^{\gamma_0}\Bigg)^{1/{\gamma_0}},
    \end{equation*}
    where $\alpha$ is the same constant as in Theorem \ref{thm v-BMO}.
\end{lemma}

With Lemmas \ref{bou |nabla v-nabla w} and \ref{boundary v-BMO} in hand, by mimicking the proof of Lemma \ref{lem nabla w infty}, we are ready to establish an a priori estimate of $\|\nabla_y w\|_{L^{\infty}(B^+_{R/2})}$ by assuming $w\in C^{0,1}(\overline{B^+_R})$. The general case follows from the approximation.
\begin{lemma}
    Let $w$ be the unique solution to \eqref{w in u_1+W}, we have
    \begin{equation}\label{bou nabla w L infty}
        \|\nabla_yw\|_{L^{\infty}(B_{R/2}^+)}\leq CR^{-n}\|\nabla_yw+s\|_{L^1(B_R^+)}\quad \mbox{if} \quad p\geq 2,
    \end{equation}
    where $C$ is a constant depending only on $n$, $p$, and $\lambda$; and
    \begin{equation}\label{bou nabla w L p<2}
        \|\nabla_yw\|_{L^{\infty}(B_{R/2}^+)}\leq CR^{-n/\gamma_0}\|\nabla_yw|+s\|_{L^{\gamma_0}(B_R^+)} \quad \mbox{if} \quad 1<p<2, \quad 0<\gamma_0<1,
    \end{equation}
    where $C$ is a constant depending on $n$, $p$, $\lambda$, and $\gamma_0$.
\end{lemma}
\begin{lemma}
    Suppose $u_1\in W^{1,p}(B_{R'}^+)$ is a solution to \eqref{equ u_1}, and let $\alpha\in(0,1)$ be the constant in Theorem \ref{thm v-BMO}. Then for any $\varepsilon\in(0,1)$ and $r\in(0,R'/4]$, we have the following estimates:
 
  {(i) when $p\geq 2$}, there exists a constant $C$ depending only $n$, $p$, and $\lambda$ such that
    \begin{align}\label{bou pro 1}
        &\inf_{\theta\in\mathbb R}\Bigg(\fint_{B_{\varepsilon r}^+}|\nabla_{y_n}u_1-\theta|+|\nabla_{y'}u_1|\Bigg)\nonumber\\
       &\leq C\varepsilon^{\alpha}\inf_{\theta\in\mathbb R}\Bigg(\fint_{B_r^+}|\nabla_{y_n}u_1-\theta|
        +|\nabla_{y'}u_1|\Bigg)
        +C\varepsilon^{-n}\Bigg(\frac{|\mu_1|(B_{2r}^+)}{r^{n-1}}\Bigg)^{\frac{1}{p-1}}\nonumber\\
         &\quad+C\varepsilon^{-n}(\varrho_1(r))^{\frac{2}{p}}\fint_{B_{2r}^+}\big(|\nabla u_1|+s\big)\ dy;
    \end{align} 
    {(ii) when $3/2\leq p<2$, take $\gamma_0=2-p$}, there exists a constant $C$ depending only $n$, $p$, and $\lambda$ such that
    \begin{align}\label{bou pro 2}
        &\inf_{\theta\in\mathbb R}\Bigg(\fint_{B_{\varepsilon r}^+}|\nabla_{y_n}u_1-\theta|^{\gamma_0}+|\nabla_{y'}u_1|^{\gamma_0}\Bigg)^{1/\gamma_0}\nonumber\\
        &\leq C\varepsilon^{\alpha}\inf_{\theta\in\mathbb R}\Bigg(\fint_{B_r^+}|\nabla_{y_n}u_1-\theta|^{\gamma_0}+|\nabla_{y'}u_1|^{\gamma_0}\Bigg)^{1/\gamma_0}+C\varepsilon^{-n/\gamma_0}\Bigg(\frac{|\mu_1|(B_{2r}^+)}{r^{n-1}}\Bigg)^{\frac{1}{p-1}}\nonumber\\
        &\quad +C\varepsilon^{-n/\gamma_0}\varrho_1(r)\Bigg(\fint_{B_{2r}^+}\big(|\nabla_yu_1|+s\big)^{2-p}\ dy\Bigg)^{\frac{1}{2-p}}\nonumber\\
        &\quad +C\varepsilon^{-n/\gamma_0}\frac{|\mu_1|(B_{2r}^+)}{r^{n-1}}\fint_{B_{2r}^+}\big(|\nabla_yu_1|+s\big)^{2-p}\ dy;
    \end{align}
    {(iii) when $1<p<3/2$, take $\gamma_0=(p-1)^2/2$}, there exists a constant $C$ depending only $n$, $p$, and $\lambda$ such that
    \begin{align}\label{bou pro 3}
        &\inf_{\theta\in\mathbb R}\Bigg(\fint_{B_{\varepsilon r}^+}|\nabla_{y_n}u_1-\theta|^{\gamma_0}+|\nabla_{y'}u_1|^{\gamma_0}\Bigg)^{1/\gamma_0}\nonumber\\
        &\leq C\varepsilon^{\alpha}\inf_{\theta\in\mathbb R}\Bigg(\fint_{B_r^+}|\nabla_{y_n}u_1-\theta|^{\gamma_0}+|\nabla_{y'}u_1|^{\gamma_0}\Bigg)^{1/\gamma_0}+C\varepsilon^{-n/\gamma_0}\Bigg(\frac{|\mu_1|(B_{2r}^+)}{r^{n-1}}\Bigg)^{\frac{1}{p-1}}\nonumber\\
        &\quad +C\varepsilon^{-n/\gamma_0}\frac{|\mu_1|(B_{2r}^+)}{r^{n-1}}\Bigg(\fint_{B_{2r}^+}\big(|\nabla u_1|+s\big)^{\gamma_0}\ dy\Bigg)^{(2-p)/\gamma_0}\nonumber\\
        &\quad +C\varepsilon^{-n/\gamma_0}\varrho_1(r)\Bigg(\fint_{B_{2r}^+}\big(|\nabla u_1|+s\big)^{\gamma_0}\ dy\Bigg)^{1/\gamma_0}.
    \end{align}
    The $\alpha$ above is the same as in Theorem \ref{thm v-BMO}.
\end{lemma}
\begin{proof}
    For a function $f\in W^{1,p}(B_{R'}^+)$ and $B_r^+\subset B_{R'}^+$, there exist $\theta_r(f)$ and $\theta_{r;\gamma_0}(f)$ such that
    \begin{equation*}
        \fint_{B_r^+}|\nabla_{y_n} f-\theta_r(f)|+|\nabla _{y'}f|=\inf_{\theta\in\mathbb R}\fint_{B_r^+}|\nabla_{y_n} f-\theta|+|\nabla_{y'}f|
    \end{equation*}
    and
    \begin{equation*}
        \fint_{B_r^+}|\nabla _{y_n}f-\theta_{r;\gamma_0}(f)|^{\gamma_0}+|\nabla _{y'}f|^{\gamma_0}=\inf_{\theta\in\mathbb R}\fint_{B_r^+}|\nabla_{y_n}f-\theta|^{\gamma_0}+|\nabla _{y'}f|^{\gamma_0}.
    \end{equation*}
    When $p\geq2$,  by Lemma \ref{boundary v-BMO} and the triangle inequality, we have
     \begin{align}\label{bou var r Du_1-theta}
         &\fint_{B_{\varepsilon r}^+}\big(|\nabla_{y_n}u_1-\theta_{\varepsilon r}(u_1)|+|\nabla_{y'}u_1|\big)
         \leq\fint_{B_{\varepsilon r}^+}\big(|\nabla_{y_n}u_1-\theta_{\varepsilon r}(v)|+|\nabla_{y'}u_1|\big)\nonumber\\
         &\leq C\fint_{B_{\varepsilon r}^+}\big(|\nabla_{y_n}u_1-\nabla_{y_n}v|+|\nabla_{y'}u_1-\nabla_{y'}v|\big)+C\fint_{B_{\varepsilon r}^+}\big(|\nabla_{y_n}v-\theta_{\varepsilon r}(v)|+|\nabla_{y'}v|\big)\nonumber\\
         &\leq C\varepsilon^{\alpha}\fint_{B_r^+}\big(|\nabla_{y_n}v-\theta_r(v)|+|\nabla_{y'}v|\big)+C\varepsilon^{-n}\fint_{B_r^+}|\nabla u_1-\nabla  v|\nonumber\\
         &\leq C\varepsilon^{\alpha}\fint_{B_r^+}\big(|\nabla_{y_n}v-\theta_r(u_1)|+|\nabla_{y'}v|\big)+C\varepsilon^{-n}\fint_{B_r^+}|\nabla u_1-\nabla v|\nonumber\\
         &\leq C\varepsilon^{\alpha}\fint_{B_r^+}\big(|\nabla_{y_n}u_1-\theta_r(u_1)|+|\nabla_{y'}u_1|\big)+C\varepsilon^{-n}\fint_{B_r^+}|\nabla u_1-\nabla v|.
     \end{align}
     Using \eqref{boundary nabla v-nabla w}, \eqref{bou nabla w L infty}, and the triangle inequality, we have
     \begin{align*}
        \fint_{B_r^+}|\nabla u_1-\nabla_yv|&\leq \fint_{B_r^+}|\nabla u_1-\nabla_y w|+\fint_{B_r^+}|\nabla w-\nabla v|\\
         &\leq\fint_{B_r^+}|\nabla u_1-\nabla w|+C(\varrho_1(r))^{\frac{2}{p}}\big(\|\nabla w\|_{L^{\infty}(B_r^+)}+s\big)\\
         &\leq C\fint_{B_{2r}^+}|\nabla u_1-\nabla w|+C(\varrho_1(r))^{\frac{2}{p}}\fint_{B_{2r}^+}\big(|\nabla w|+s\big)\\
         &\leq C\fint_{B_{2r}^+}|\nabla u_1-\nabla w|+C(\varrho_1(r))^{\frac{2}{p}}\fint_{B_{2r}^+}\big(|\nabla u_1|+s\big).
     \end{align*}
     Combining the last inequality with \eqref{bou nabla u_1-nabla w} and \eqref{bou var r Du_1-theta}, we get \eqref{bou pro 1}.
     
     When $3/2\leq p<2$ and $1<p<3/2$, by a similar argument as in the proof of \eqref{bou pro 1}, and using Lemma \ref{boundary v-BMO}, \eqref{bou nabla u1-nabla w gamma p>3/2}, \eqref{bou nabla u1-nab w gamma p<3/2}, \eqref{boundary |nabla v-nabla w| gamma} and \eqref{bou nabla w L p<2}, we derive \eqref{bou pro 2} and \eqref{bou pro 3}. The details are omitted here.
\end{proof}
     We now define
     \begin{equation*}
         \overline{\phi}(x,r)=\inf_{\theta\in\mathbb R}\big(\fint_{\Omega_r(x)}|\nabla_nu-\theta|+|\nabla_{x'}u|\big),
     \end{equation*}
     and 
     \begin{equation*}
         \overline{\psi}(x,r)=\inf_{\theta\in\mathbb  R}\big(\fint_{\Omega_r(x)}|\nabla_nu-\theta|^{\gamma_0}+|\nabla_{x'}u|^{\gamma_0}\big)^{1/\gamma_0}.
     \end{equation*}

     \begin{corollary}\label{cor bou}
    Let $u\in W_0^{1,p}(\Omega)$ be a solution of \eqref{bou u=0} and $x_0\in\partial\Omega$. Then for any $\varepsilon\in(0,1/4)$, $r\leq R'/2$, there exist some constant $C$ depending $n$, $p$, and $\lambda$ such that the following assertions hold:
    
    (i) when $p\geq 2$, then
    \begin{align}\label{bou cor p>2}
        \overline{\phi}(x_0,\varepsilon r)\leq C\varepsilon^{\alpha}\overline{\phi}(x_0,r)+C\varepsilon^{-n}\Bigg(\frac{|\mu|(\Omega_{2r}(x_0))}{r^{n-1}}\Bigg)^{\frac{1}{p-1}}+C\varepsilon^{-n}(\varrho_1(r))^{\frac{2}{p}}\fint_{\Omega_{2r}(x_0)}\big(|\nabla u|+s\big);
    \end{align}

    (ii) when $3/2\leq p<2$ and $\gamma_0=2-p$, then
    \begin{align}\label{bou cor 3/2<p}
        \overline{\psi}(x_0,\varepsilon r)&\leq C\varepsilon^{\alpha}\overline{\psi}(x_0,r)+C\varepsilon^{-n/\gamma_0}\Bigg(\frac{|\mu|(\Omega_{2r}(x_0))}{r^{n-1}}\Bigg)^{\frac{1}{p-1}}\nonumber\\
        &\quad +C\varepsilon^{-n/\gamma_0}\varrho_1(r)\Bigg(\fint_{\Omega_{2r}(x_0)}\big(|\nabla u|+s\big)^{2-p}\Bigg)^{\frac{1}{2-p}}\nonumber\\
        &\quad +C\varepsilon^{-n/\gamma_0}\frac{|\mu|(\Omega_{2r}(x_0))}{r^{n-1}}\fint_{\Omega_{2r}(x_0)}\big(|\nabla u|+s\big)^{2-p};
    \end{align}

    (iii) when $1<p<3/2$ and $\gamma_0=(p-1)^2/2$, then
    \begin{align*}
        \overline{\psi}(x_0,\varepsilon r)&\leq C\varepsilon^{\alpha}\overline{\psi}(x_0,r)+C\varepsilon^{-n/\gamma_0}\Bigg(\frac{|\mu|(\Omega_{2r}(x_0))}{r^{n-1}}\Bigg)^{\frac{1}{p-1}}\nonumber\\
        &\quad +C\varepsilon^{-n/\gamma_0}\varrho_1(r)\Bigg(\fint_{\Omega_{2r}(x_0)}\big(|\nabla u|+s\big)^{\gamma_0}\Bigg)^{1/\gamma_0}\nonumber\\
        &\quad +C\varepsilon^{-n/\gamma_0}\frac{|\mu|(\Omega_{2r}(x_0))}{r^{n-1}}\Bigg(\fint_{\Omega_{2r}(x_0)}\big(|\nabla u|+s\big)^{\gamma_0}\Bigg)^{(2-p)/\gamma_0}.
    \end{align*}
\end{corollary}
\begin{proof}
When $p\geq2$, letting $\varepsilon\in(0,1)$ and $r\in(0,R'/4)$, by using the method of changing variables as well as \eqref{Omeag in Gamma in  Omeag}, \eqref{c1<Omega<c2}, \eqref{sup chi-chi}, and the triangle inequality, we have
\begin{align*}
    &\inf_{\theta\in\mathbb R}\Bigg(\fint_{B_{\varepsilon r}^+}|\nabla_{y_n}u_1-\theta|+|\nabla_{y'}u_1|\Bigg)\nonumber\\
    &=\inf_{\theta\in\mathbb R}\Bigg(\fint_{\Gamma(B_{\varepsilon r}^+)}|\nabla_nu-\theta|+|\nabla_nu\nabla_{x'}\chi+\nabla_{x'}u|\Bigg)\nonumber\\
    &\geq C\inf_{\theta\in\mathbb R}\Bigg(\fint_{\Omega_{\varepsilon r/2}}|\nabla_nu-\theta|+|\nabla_{x'}u|\Bigg)-C_1\fint_{\Omega_{\varepsilon r/2}}|\nabla_nu\nabla_{x'}\chi|\nonumber\\
    &\geq C\overline{\phi}(0,\varepsilon r/2)-C_1\varrho_1(\varepsilon r/2)\fint_{\Omega_{\varepsilon r/2}}|\nabla u|,
\end{align*}
where both $C$ and $C_1$ are non-negative constants depending only on $n$. On the other hand,
\begin{align*}
   & \inf_{\theta\in\mathbb R}\Bigg(\fint_{B_r^+}|\nabla_{y_n}u_1-\theta|+|\nabla_{y'}u_1|\Bigg)\nonumber\\
    &=\inf_{\theta\in\mathbb R}\Bigg(\fint_{\Gamma(B_r^+)}|\nabla_nu-\theta|+|\nabla_nu\nabla_{x'}\chi+\nabla_{x'}u|\Bigg)\nonumber\\
    &\leq C_2 \inf_{\theta\in\mathbb R}\Bigg(\fint_{\Omega_{2r}}|\nabla_nu-\theta|+|\nabla_{x'}u|\Bigg)+C_2\fint_{\Omega_{2r}}|\nabla_nu\nabla_{x'}\chi|\nonumber\\
    &\leq C_2 \overline{\phi}(0,2r)+C_2\varrho_1(2r)\fint_{\Omega_{2r}}|\nabla u|,
\end{align*}
where $C_2$ is a non-negative constant depending only on $n$. Using \eqref{chi <1/2} and \eqref{c1<Omega<c2}, combining the last two inequalities and \eqref{bou pro 1} yields
\begin{equation}\label{overline phi(0,var r/2)}
    \overline{\phi}(0,\varepsilon r/2)\leq C\varepsilon^{\alpha}\overline{\phi}(0,2r)+C\varepsilon^{-n}\Bigg(\frac{|\mu|(\Omega_{4r})}{r^{n-1}}\Bigg)^{\frac{1}{p-1}}+C\varepsilon^{-n}(\varrho_1(2r))^{\frac{2}{p}}\fint_{\Omega_{4r}}\big(|\nabla u|+s\big)\ dx.
\end{equation}

When $3/2\leq p<2$, $\gamma_0=2-p$, using \eqref{chi <1/2}, \eqref{c1<Omega<c2}, and \eqref{bou pro 2} yields 
\begin{align}\label{overline psi(0,var r/2)P>3/2}
  \overline{\psi}(0,\varepsilon r/2)&\leq C\varepsilon^{\alpha} \overline{\psi}(0,2r)+C\varepsilon^{-n/\gamma_0}\Bigg(\frac{|\mu|(\Omega_{4r})}{r^{n-1}}\Bigg)^{\frac{1}{p-1}} \nonumber \\
  &\quad+C\varepsilon^{-n/\gamma_0}\varrho_1(2r)\Bigg(\fint_{\Omega_{4r}}\big(|\nabla u|+s\big)^{2-p}\ dx\Bigg)^{\frac{1}{2-p}}\nonumber\\
  &\quad+C\varepsilon^{-n/\gamma_0}\frac{|\mu|(\Omega_{4r})}{r^{n-1}}\fint_{\Omega_{4r}}\big(|\nabla u|+s\big)^{2-p}\ dx.
\end{align}

When $1<p<3/2$, $\gamma_0=(p-1)^2/2$, using \eqref{chi <1/2}, \eqref{c1<Omega<c2}, and \eqref{bou pro 3} gives 
\begin{align}\label{overline psi(0,var r/2)p<3/2}
    \overline{\psi}(0,\varepsilon r/2)&\leq C\varepsilon^{\alpha}\psi(0,2r)+C\varepsilon^{-n/\gamma_0}\Bigg(\frac{|
    \mu|(\Omega_{4r})}{r^{n-1}}\Bigg)^{\frac{1}{p-1}}\nonumber\\  
    &\quad +C\varepsilon^{-n/\gamma_0}\varrho_1(2r)\Bigg(\fint_{\Omega_{4r}}\big(|\nabla u|+s\big)^{\gamma_0}\ dx\Bigg)^{1/\gamma_0}\nonumber\\
    &\quad +C\varepsilon^{-n/\gamma_0}\frac{|\mu|(\Omega_{4r})}{r^{n-1}}\Bigg(\fint_{\Omega_{4 r}}\big(|\nabla u|+s\big)^{\gamma_0}\ d x\Bigg)^{(2-p)/\gamma_0}.
\end{align}
By replacing $\varepsilon/4$ and $2r$ with $\varepsilon$ and $r$ for \eqref{overline phi(0,var r/2)}, \eqref{overline psi(0,var r/2)P>3/2}, and \eqref{overline psi(0,var r/2)p<3/2}, respectively, we finish the proof of Corollary \ref{cor bou}.
\end{proof}

\subsection{Global pointwise gradient estimates}
Next as in Subsection \ref{sub-auxiliary-est}, we  define
\begin{equation*}
    \phi(x,\rho)=\inf_{{\bf q}\in\mathbb R^n}\fint_{\Omega_{\rho}(x)}|\nabla u-{\bf q}|\quad \text{and}\quad\psi(x,\rho)=\inf_{{\bf q}\in \mathbb R^n}\Bigg(\fint_{\Omega_{\rho}(x)}|\nabla u-q|^{\gamma_0}\Bigg)^{1/\gamma_0},
\end{equation*}
where $0<\gamma_0<1$,
 and choose ${\bf q}_{x,r}\in\mathbb R^n$ and ${\bf q}_{x,r;\gamma_0}\in \mathbb R^n$ such that
\begin{equation*}
    \fint_{\Omega_r(x)}|\nabla u-{\bf q}_{x,r}|=\inf_{{\bf q}\in \mathbb R^n}\fint_{\Omega_r(x)}|\nabla u-{\bf q}|
\end{equation*}
and 
\begin{equation*}
    \Bigg(\fint_{\Omega_r(x)}|\nabla u-{\bf q}_{x,r}|^{\gamma_0}\Bigg)^{1/\gamma_0}=\inf_{{\bf q}\in\mathbb R^n}\Bigg(\fint_{\Omega_r(x)}|\nabla u-{\bf q}|^{\gamma_0}\Bigg)^{1/\gamma_0}.
\end{equation*}

We then prove the global pointwise gradient estimate of $\nabla u$.
For this, we let $\varepsilon=\varepsilon(n, p, \lambda, \alpha)\in(0,1/4)$ be sufficiently small such that $C\varepsilon^{\alpha}\leq 1/4 $, where $C$ is the constant in Corollary \ref{cor bou} and $\alpha\in(0,1)$ is from Theorem \ref{thm v-BMO}.

 Fix a point $x_0\in\partial\Omega$ and $R\leq R'/2$. For any non-negative integer $j$, set $r_j=\varepsilon^jR$, $\Omega_j=\Omega_{2r_j}(x_0)$,
\begin{equation*}
    S_j=\fint_{\Omega_j}\big(|\nabla u|+s\big)\ dx,\quad \phi_j=\phi(x_0,r_j),\quad and \quad\overline{\phi}_j=\overline{\phi}(x_0,r_j).
\end{equation*}
When $p\geq 2$, using \eqref{bou cor p>2}, we get
\begin{equation*}
    \overline{\phi}_{j+1}\leq \frac{1}{4}\overline{\phi}_j+C\Bigg(\frac{|\mu|(\Omega_j)}{r^{n-1}_j}\Bigg)^{\frac{1}{p-1}}+C(\varrho_1(r))^{\frac{2}{p}}S_j.
\end{equation*}
Note that $\phi_j\leq\overline{\phi}_j\leq CS_j$. Let $j_0$ and $d$ be nonnegative integers and $j_0\leq d$, and sum the above inequality with respect to $j=j_0, j_0+1, ... , d$, yields
\begin{equation}\label{bou sum phi}
    \sum_{j=j_0}^{d+1}\phi_j\leq\sum_{j=j_0}^{d+1}\overline{\phi}_j\leq CS_{j_0}+C\sum_{j=j_0}^d\Bigg(\frac{|\mu|(\Omega_j)}{r_j^{n-1}}\Bigg)^{\frac{1}{p-1}}+C\sum_{j=j_0}^d(\varrho_1(r_j))^{\frac{2}{p}}S_j,
\end{equation}
for any $x_0\in\partial\Omega$ and $R\leq R'/2$. On the other hand, for any $x_0\in\Omega$ and $R>0$ such that $r_{j_0}=\varepsilon^{j_0}R<\text{dist}(x_0,\partial\Omega)/2$, using \eqref{d+1 phi j}, the following holds
\begin{equation}\label{sum phi bou in}
    \sum_{j=j_0}^{d+1}\phi_j\leq C\phi_{j_0}+C\sum_{j=j_0}^d\Bigg(\frac{|\mu|(\Omega_j)}{r_j^{n-1}}\Bigg)^{\frac{1}{p-1}}+C\sum_{j=j_0}^d(\varrho(r_j))^{\frac{2}{p}}S_j.
\end{equation}

When $3/2\leq p<2$, we take $\gamma_0=2-p$  and define 
\begin{equation*}
    T_j=\Bigg(\fint_{\Omega_j}\big(|\nabla u|+s\big)^{2-p}\Bigg)^{\frac{1}{2-p}}, \quad \psi_j=\psi(x_0,r_j), \quad and \quad \overline{\psi}_j=\overline{\psi}(x_0,r_j).
\end{equation*}
Using \eqref{bou cor 3/2<p}, for any $x_0\in\partial\Omega$ and $R\leq R'/2$, we get
\begin{align}\label{bou sum psi p>3/2}
    \sum_{j=j_0}^{d+1}\psi_j\leq C\sum_{j=j_0}^{d+1}\overline{\psi}_j&\leq CT_{j_0}+C\sum_{j=j_0}^d\Bigg(\frac{|\mu|(\Omega_j)}{r_j^{n-1}}\Bigg)^{\frac{1}{p-1}}\nonumber\\
    &\quad +C\sum_{j=j_0}^d\frac{|\mu|(\Omega_j)}{r_j^{n-1}}T_j^{2-p}+C\sum_{j=j_0}^d\varrho_1(r_j)T_j.
\end{align}
Moreover, for any $x_0\in\Omega$ and $R>0$ such that $r_{j_0}=\varepsilon^{j_0}R<\text{dist}(x_0,\partial\Omega)/2$, we obtain
\begin{align}\label{sum psi bou in p>3/2}
    \sum_{j=j_0}^{d+1}\psi_j\leq C\psi_{j_0}+C\sum_{j=j_0}^d\frac{|\mu|(\Omega_j)}{r^{n-1}_j}\Bigg)^{\frac{1}{p-1}}+C\sum_{j=j_0}^d\frac{|\mu|(\Omega_j)}{r_j^{n-1}}T_j^{2-p}+C\sum_{j=j_0}^d\varrho(r_j)T_j.
\end{align}

When $1<p<3/2$, we take $\gamma_0=(p-1)^2/2$ and set 
\begin{equation*}
    U_j=\Bigg(\fint_{\Omega_j}\big(|\nabla u|+s\big)^{(p-1)^2/2}\Bigg)^{2/(p-1)^2},\quad \psi_j=\psi(x_0,r_j), \quad and \quad\overline{\psi}_j=\overline{\psi}(x_0,r_j).
\end{equation*}
Similarly, we obtain that for any $x_0\in\partial\Omega$ and $R\leq R'/2$,
\begin{align}\label{bou sum psi p<3/2}
    \sum_{j=j_0}^{d+1}\psi_j\leq C\sum_{j=j_0}^{d+1}\overline{\psi}_j&\leq CU_{j_0}+C\sum_{j=j_0}^d\Bigg(\frac{|\mu|(\Omega_j)}{r_j^{n-1}}\Bigg)^{\frac{1}{p-1}}\nonumber\\
    &\quad +C\sum_{j=j_0}^d\frac{|\mu|(\Omega_j)}{r_j^{n-1}}U_j^{2-p}+C\sum_{j=j_0}^d\varrho_1(r_j)U_j
\end{align}
and 
\begin{align}\label{sum psi bou in p<3/2}
    \sum_{j=j_0}^{d+1}\psi_j\leq C\psi_{j_0}+C\sum_{j=j_0}^d\frac{|\mu|(\Omega_j)}{r^{n-1}_j}\Bigg)^{\frac{1}{p-1}}+C\sum_{j=j_0}^d\frac{|\mu|(\Omega_j)}{r_j^{n-1}}U_j^{2-p}+C\sum_{j=j_0}^d\varrho(r_j)U_j
\end{align}
for any $x_0\in\Omega$ and $R>0$ such that $r_{j_0}=\varepsilon^{j_0}R<\text{dist}(x_0,\partial\Omega)/2$ holds.

Now we define
\begin{align*}
    &\Omega'_j=\Omega_{8r_j}(x_0),\quad S'_j=\fint_{\Omega'_j}\big(|\nabla u|+s\big)\ dx,\quad
    T_j'=\Bigg(\fint_{\Omega'_j}\big(|\nabla u|+s\big)^{2-p}\Bigg)^{\frac{1}{2-p}},\\\quad &and\quad U_j'=\Bigg(\fint_{\Omega_j'}\big(|\nabla u|+s\big)^{\frac{(p-1)^2}{2}}\Bigg)^{\frac{2}{(p-1)^2}}.
\end{align*}
Thus we have the following lemma.
\begin{lemma}\label{lemma global sum phi+psi}
    Let $u\in W_0^{1,p}(\Omega)$ be a solution of \eqref{bou u=0}. For any fixed $x_0\in\bar{\Omega}$ and $R\leq R'/2$, there exists some positive constant $C$ depending on $n$, $p$, and $\lambda$ such that we have the assertions as follows:

    (i) when $p\geq 2$, then
    \begin{equation}\label{bou P>2 sum phi j}
        \sum_{j=j_0}^{d+1}\phi_j\leq CS'_{j_0}+C\sum_{j=j_0}^d\Bigg(\frac{|\mu|(\Omega'_j)}{r_j^{n-1}}\Bigg)^{\frac{1}{p-1}}+C\sum_{j=j_0}^{d+1}(\varrho_1(r_j))^{\frac{2}{p}}S'_j;
    \end{equation}

    (ii) when $3/2\leq p<2$, then 
    \begin{align}\label{bou p>3/2 sum psi j}
        \sum_{j=j_0}^{d+1}\psi_j\leq CT'_{j_0}+C\sum_{j=j_0}^d\Bigg(\frac{|\mu|(\Omega'_j)}{r_j^{n-1}}\Bigg)^{\frac{1}{p-1}}+C\sum_{j=j_0}^d\frac{|\mu|(\Omega'_j)}{r_j^{n-1}} {T'_j}^{2-p}+C\sum_{j=j_0}^{d+1}\varrho_1(r_1)T'_j;
    \end{align}

    (iii) when $1<p<3/2$, then
    \begin{equation}\label{bou p<3/2 sum psi j}
        \sum_{j=j_0}^{d+1}\psi_j\leq CU'_{j_0}+C\sum_{j=j_0}^d\Bigg(\frac{|\mu|(\Omega'_j)}{r_j^{n-1}}\Bigg)^{\frac{1}{p-1}}+C\sum_{j=j_0}^d\frac{|\mu|(\Omega'_j)}{r_j^{n-1}} {U'_j}^{2-p}+C\sum_{j=j_0}^{d+1}\varrho_1(r_1)U'_j.
    \end{equation}
\end{lemma}
\begin{proof}
    For any fixed $x_0\in \partial\Omega$, since $\Omega_j\subset\Omega'_j$, we have $S_j\leq CS'_j$. Then it follows from \eqref{bou sum phi}, \eqref{bou sum psi p>3/2}, and \eqref{bou sum psi p<3/2} that \eqref{bou P>2 sum phi j}-\eqref{bou p<3/2 sum psi j} hold. For $x_0\in\Omega$, combining \eqref{sum phi bou in}, \eqref{sum psi bou in p>3/2}, and \eqref{sum psi bou in p<3/2}, we conclude that it suffices to prove \eqref{bou P>2 sum phi j}-\eqref{bou p<3/2 sum psi j} when $r_{j_0}\geq \text{dist}(x_0,\partial\Omega)/2$. Let $r_{j_1}\geq\text{dist}(x_0,\partial\Omega)/2$ and $r_{j_1+1}<\text{dist}(x_0,\partial\Omega)/2$, by \eqref{sum phi bou in}, we have 
    \begin{equation}\label{bou phi j sum j1+1  d+1}
        \sum_{j=j_1+1}^{d+1}\phi_j<C\phi_{j_1+1}+C\sum_{j=j_1+1}^d\Bigg(\frac{|\mu|(\Omega_j)}{r_j^{n-1}}\Bigg)^{\frac{1}{p-1}}+C\sum_{j=j_1+1}^d(\varrho(r_j))^{\frac{2}{p}}S_j.
    \end{equation}
    According to \eqref{c1<Omega<c2}, we also have
    \begin{equation}\label{bou phi j1+1}
        \phi_{j_1+1}\leq \fint_{\Omega_{r_{j_1+1}}(x_0)}|\nabla u-{\bf q}_{x_0,r_{j_1}}|\leq C\phi_{j_1}.
    \end{equation}
    
    For any $j\in$\{$j_0, j_0+1,...,j_1$\}, $r_j\geq \text{dist}(x_0,\partial\Omega)/2$, choose $y_0\in\partial\Omega$ such that $\text{dist}(x_0, \partial\Omega)=|y_0-x_0|$. Thus $\Omega_{r_j}(x_0)\subset\Omega_{3r_j}(y_0)$ and $\Omega_{6r_j}(y_0)\subset\Omega_{8r_j}(x_0)$. For simplicity, we define
    \begin{equation*}
        S''_j:=\fint_{\Omega_{6r_j}(y_0)}\big(|\nabla u|+s\big)\ dx.
    \end{equation*}
    Using  \eqref{bou sum phi} yields
    \begin{align}\label{bou sum phij j0 j1}
        \sum_{j=j_0}^{j_1}\phi_j&\leq C\sum_{j=j_0}^{j_1}\phi(y_0,3r_j)\nonumber\\
        &\leq CS''_{j_0}+C\sum_{j=j_0}^{j_1}\Bigg(\frac{|\mu|(\Omega_{6r_j}(y_0))}{r_j^{n-1}}\Bigg)^{\frac{1}{p-1}}+C\sum_{j=j_0}^{j_1+1}(\varrho_1(3r_j))^{\frac{2}{p}}S''_j\nonumber\\
        &\leq CS'_{j_0}+C\sum_{j=j_0}^{j_1}\Bigg(\frac{|\mu|(\Omega'_j)}{r_j^{n-1}}\Bigg)^{\frac{1}{p-1}}+C\sum_{j=j_0}^{j_1+1}(\varrho_1(r_j))^{\frac{2}{p}}S'_j,
    \end{align}
     Combining \eqref{bou phi j sum j1+1  d+1}, \eqref{bou phi j1+1}, and \eqref{bou sum phij j0 j1}, we obtain \eqref{bou P>2 sum phi j}.
    
   Similarly, by the same argument as above, we obtain \eqref{bou p>3/2 sum psi j} and \eqref{bou p<3/2 sum psi j}. Lemma \ref{lemma global sum phi+psi} is proved.
\end{proof}
\begin{proof}[Proof of Theorem \ref{bou pointwise estiamte}]
    In view of the same ideas as in the proof of Theorem \ref{main result Pwg}, and using Lemma \ref{lemma global sum phi+psi}, we complete the proof of Theorem \ref{bou pointwise estiamte}. We thus omit the details.
\end{proof}
\begin{proof}[Proof of the Corollary \ref{Omega bounded}.]
By using $u$ as a test function for $\eqref{bou u=0}$ and imitating the proof of \eqref{inequ Dum}, we obtain
    \begin{equation*}
        \|\nabla u\|_{L^p(\Omega)}\leq C\|\mu\|_{W^{-1,\frac{p}{p-1}}(\mathbb R^n)}^{\frac{1}{p-1}}.
    \end{equation*}
Then combining with H\"{o}lder inequality, we have
\begin{equation}\label{L1-Lp}
\|\nabla u\|_{L^1(\Omega)}\leq C\|\nabla u\|_{L^p(\Omega)}\leq C\|\mu\|_{W^{-1,\frac{p}{p-1}}(\mathbb R^n)}^{\frac{1}{p-1}}.
\end{equation}
It follows from \cite[Theorem 1]{MR1983} and Young's inequality that
    \begin{equation*}
        \|\mu\|_{W^{-1,\frac{p}{p-1}}(\mathbb R^n)}\leq C\Bigg(\int_{\mathbb R^n}{\bf W}_{\frac{1}{p},p}^1(|\mu|)\ d|\mu|\Bigg)^{\frac{p-1}{p}}\leq C\Bigg(\int_{\mathbb R^n}{\bf W}_{\frac{1}{p},p}^1(|\mu|)\ d|\mu|\Bigg)^{p-1}+C.
    \end{equation*}
 This together with \eqref{L1-Lp} and \eqref{bou p>2 point} yields 
\begin{equation}\label{Du-1}
        \|\nabla u\|_{L^{\infty}(\Omega)}\leq C\|{\bf W}_{1/p,p}^1(|\mu|)\|_{L^{\infty}(\Omega)}+Cs+C.
    \end{equation}
Denote $v:=\lambda^{-1}u$, where $\lambda>0$ is a constant determined later. Then $v$ satisfies 
\begin{equation*}
-\Div\big(a(x)(|\nabla v|^2+(\lambda^{-1}s)^2)^{\frac{p-2}{2}}\nabla v\big)=\lambda^{1-p}\mu\quad\mbox{in}~\Omega,
\end{equation*}
and from \eqref{Du-1}, we have
\begin{equation*}
        \|\nabla v\|_{L^{\infty}(\Omega)}\leq C\lambda^{-1}\|{\bf W}_{1/p,p}^1(|\mu|)\|_{L^{\infty}(\Omega)}+C\lambda^{-1}s+C.
    \end{equation*}
This gives
\begin{equation*}
        \|\nabla u\|_{L^{\infty}(\Omega)}\leq C\|{\bf W}_{1/p,p}^1(|\mu|)\|_{L^{\infty}(\Omega)}+Cs+C\lambda.
    \end{equation*}
By choosing $\lambda=\|{\bf W}_{1/p,p}^1(|\mu|)\|_{L^{\infty}(\Omega)}$, we obtain \eqref{Linftyp2}. Similarly, by using Theorem \ref{bou pointwise estiamte} and \eqref{ine-1p2}, we obtain \eqref{Linfty1p2} and finish the proof of Corollary \ref{Omega bounded}.
\end{proof}

\subsection{Global gradient modulus of continuity estimates}
Let $x_0\in\bar{\Omega}$ and $0<R<R'$. For any fixed $\alpha_1\in(0,\alpha)$, take $\alpha_2=(\alpha+\alpha_1)/2$ and choose $\varepsilon=\varepsilon(n, p, \lambda, \alpha, \alpha_1)\in(0,1/4)$ sufficiently small such that $\varepsilon^{\alpha_2}<1/4$ and $C\varepsilon^{\alpha-\alpha_2}<1$ for both constants $C$ in Proposition \ref{prop-phi-Du} and Corollary \ref{cor bou}.

We now define
\begin{equation*}
    h_1(x,r)=\Bigg(\frac{|\mu|(B_r(x)\cap B_{R/2}(x_0))}{r^{n-1}}\Bigg)^{\frac{1}{p-1}},\quad g_1(x,r)=h_1(x,r)^{p-1},
\end{equation*}
and
\begin{align*}
&\hat{h}_1(x,t)=\sum_{i=1}^{\infty}\varepsilon^{\alpha_2i}h_1(x,\varepsilon^{-i}t),\quad \hat{g}_1(x,t)=\sum_{i=1}^{\infty}\varepsilon^{\alpha_2i}g_1(x,\varepsilon^{-i}t),\\
&\overline{h}_1(x,t)=\sum_{i=1}^{\infty}\varepsilon^{\alpha_1i}h_1(x,\varepsilon^{-i}t),\quad \overline{g}_1(x,t)=\sum_{i=1}^{\infty}\varepsilon^{\alpha_1i}g_1(x,\varepsilon^{-i}t),\\
&\hat{\varrho}_1(t)=\sum_{i=1}^{\infty}\varepsilon^{\alpha_2i}\big(\varrho_1(\varepsilon^{-i}t)^{\frac{2}{p}}[\varepsilon^{-i}t\leq R/2]+\varrho_1(R/2)^{\frac{2}{p}}[\varepsilon^{-i}t>R/2]\big),\\
&\check{\varrho}_1(t)=\sum_{i=1}^{\infty}\varepsilon^{\alpha_2i}\big(\varrho_1(\varepsilon^{-i}t)[\varepsilon^{-i}t\leq R/2]+\varrho_1(R/2)[\varepsilon^{-i}t>R/2]\big),\\
&\tilde{\varrho}_1(t)=\sum_{i=1}^{\infty}\varepsilon^{\alpha_1i}\big(\varrho_1(\varepsilon^{-i}t)^{\frac{2}{p}}[\varepsilon^{-i}t\leq R/2]+\varrho_1(R/2)^{\frac{2}{p}}[\varepsilon^{-i}t>R/2]\big),\\
&\overline{\varrho}_1(t)=\sum_{i=1}^{\infty}\varepsilon^{\alpha_1i}\big(\varrho_1(\varepsilon^{-i}t)[\varepsilon^{-i}t\leq R/2]+\varrho_1(R/2)[\varepsilon^{-i}t>R/2]\big),
\end{align*}
where we used the Iverson bracket notation, i.e., $[P]=1$ if $P$ is true and $[P]=0$ otherwise.

By Corollary \ref{cor bou} and replicating the arguments as in the proof of Lemma \ref{lem-phi-w}, for any $x\in\partial\Omega$, $B_{2r}(x)\subset B_{R/2}$ and $0<\rho\leq r$, we have
\begin{align}\label{bou overline phi rho}
    \overline{\phi}(x,\rho)\leq C\big(\frac{\rho}{r}\big)^{\alpha_2}\overline{\phi}(x,r)+C\hat{h}_1(x,2\rho)+C\hat{\varrho}_1(\rho)\big(\|\nabla u\|_{L^{\infty}(\Omega_{2r}(x))}+s\big),\quad if \ p\geq 2;
\end{align}
and
\begin{align}\label{bou overline psi rho}
    \overline{\psi}(x,\rho)&\leq C\big(\frac{\rho}{r}\big)^{\alpha_2}\overline{\psi}(x,r)+C\hat{h}_1(x,2\rho)+C\hat{g}_1(x,2\rho)\big(\|\nabla u\|_{L^{\infty}(\Omega_{2r}(x))}+s\big)^{2-p}\nonumber\\
    &\quad +C\check{\varrho}_1(\rho)\big(\|\nabla u\|_{L^{\infty}(\Omega_{2r}(x))}+s\big),\quad \mbox{if} \ 1<p<2.
\end{align}
Similarly, for any $B_{2r}(x)\subset\subset\Omega$, $B_{2r}(x)\subset B_{R/2}(x_0)$ with $\rho\in(0,r]$, since $\varrho\leq \varrho_1$, by \eqref{iteration-u}, we have
\begin{align}\label{in phi rho}
    \phi(x,\rho)\leq C\big(\frac{\rho}{r}\big)^{\alpha_2}\phi(x,r)+C\hat{h}_1(x,2\rho)+C\hat{\varrho}_1(\rho)\big(\|\nabla u\|_{L^{\infty}(\Omega_{2r}(x))}+s\big)\quad \mbox{if} \ p\geq 2;
\end{align}
and
\begin{align}\label{in psi rho}
    \psi(x,\rho)&\leq C\big(\frac{\rho}{r}\big)^{\alpha_2}\psi(x,r)+C\hat{h}_1(x,2\rho)+C\hat{g}_1(x,2\rho)\big(\|\nabla u\|_{L^{\infty}(\Omega_{2r}(x))}+s\big)^{2-p}\nonumber\\
    &\quad +C\breve{\varrho}_1(\rho)\big(\|\nabla u\|_{L^{\infty}(\Omega_{2r}(x))}+s\big),\quad \mbox{if}\ 1<p<2.
\end{align}
Combining \eqref{bou overline phi rho}--\eqref{in psi rho},  we have the following lemma.
\begin{lemma}\label{Lem in+bou sum phi/psi}
    Assuming $x\in\bar{\Omega}$ and $B_{2r}(x)\subset B_{R/2}(x_0)$, there exists a constant $C$ such that for any $<\rho\leq r\leq R'$, the following assertion hold:
    
    (i) when $p\geq 2$, then
    \begin{align}\label{in+bou phi}
        \phi(x,\rho)\leq C\big(\frac{\rho}{r}\big)^{\alpha_2}r^{-n}\|\nabla u\|_{L^1(\Omega_r(x))}+C\overline{h}_1(x,\rho)+C\tilde\varrho_1(\rho)\big(\|\nabla u\|_{L^{\infty}(\Omega
        _{2r}(x))}+s\big)
    \end{align}
    and
    \begin{align}\label{in+bou sum phi}
      \sum_{j=0}^{\infty}\phi(x,\varepsilon^j\rho)&\leq C\big(\frac{\rho}{r}\big)^{\alpha_2}r^{-n}\|\nabla u\|_{L^{1}(\Omega_r(x))}+C\int_0^{\rho}\frac{\overline{h}_1(x,t)}{t} \ dt\nonumber\\
      &\quad +C\big(\|\nabla u\|_{L^{\infty}(\Omega_{2 r}(x))}+s\big)\int_0^{\rho}\frac{\tilde{\varrho}_1(t)}{t}\ dt,  
    \end{align}
    where $C$ depends $n$, $p$, $\varepsilon$, $\lambda$, and $\alpha_1$;

    (ii) when $1<p<2$ and $0<\gamma_0<1$, then
    \begin{align*}
        \psi(x,\rho)&\leq C\big(\frac{\rho}{r}\big)^{\alpha_2}r^{-n/\gamma_0}\|\nabla u\|_{L^{\gamma_0}(\Omega_r(x))}+C\overline{h}_1(x,\rho)\nonumber\\
        &\quad+C\overline{g}_1(x,\rho)\big(\|\nabla u\|_{L^{\infty}(\Omega_{2r}(x))}+s\big)
        +C\tilde\varrho_1(\rho)\big(\|\nabla u\|_{L^{\infty}(\Omega
        _{2r}(x))}+s\big)
    \end{align*}
    and
    \begin{align*}
        \sum_{j=0}^{\infty}\psi(x,\varepsilon^j\rho)&\leq C\big(\frac{\rho}{r}\big)^{\alpha_2}r^{-n/\gamma_0}\|\nabla u\|_{L^{\gamma_0}(\Omega_r(x))}+C\int_0^{\rho}\frac{\overline{h}_1(x,t)}{t}\ dt\nonumber\\
        &\quad +C\big(\|\nabla u\|_{L^{\infty}(\Omega_{2r}(x))}+s\big)^{2-p}\int_0^{\rho}\frac{\overline{g}_1(x,t)}{t}\ dt\nonumber\\
        &\quad+C\big(\|\nabla u\|_{L^{\infty}(\Omega_{2r}(x))}+s\big)\int_0^{\rho}\frac{\overline{\varrho}_1(t)}{t}\ dt,
    \end{align*}
    where $C$ depends $n$, $p$, $\varepsilon$, $\lambda$, $\alpha_1$ and $\gamma_0$.
\end{lemma}
\begin{proof}
    For simplicity, we assume that $x=0$ and denote 
    \begin{equation*}
        \phi(r):=\phi(0,r), \quad \hat{h}_1(r):=\hat{h}_1(0,r), \quad and \quad \overline{h}_1(r):=\overline{h}_1(0,r).
    \end{equation*}
    In the following, we shall prove the case of $p\geq 2$ as an example since $1<p<2$ is similar.
    
     When $r/16\leq \rho\leq r$,  by the definition of $\phi$, we get
    \begin{equation*}
        \phi(\rho)\leq C\big(\frac{\rho}{r}\big)^{\alpha_2}r^{-n}\|\nabla u\|_{L^1(\Omega_r)}.
    \end{equation*}

    When $0<\rho<r/16$, we proceed the proof according to the following three cases:

    {\bf Case (i):} $\text{dist}(0,\partial\Omega)\geq r/4$.  It follows from $B_{4r_1}\subset\Omega$ with $r_1=r/16$ and \eqref{in phi rho} that
    \begin{align*}
        \phi(\rho)&\leq C\big(\frac{\rho}{r_1}\big)^{\alpha_2}\phi(r_1)+C\hat{h}_1(2\rho)+C\hat{\varrho}_1(\rho)\big(\|\nabla u\|_{L^{\infty}(\Omega_{2r})}+s\big)\\
        &\leq C\big(\frac{\rho}{r}\big)^{\alpha_2}r^{-n}\|\nabla u\|_{L^1(\Omega_r)}+C\hat{h}_1(2\rho)+C\hat{\varrho}_1(\rho)\big(\|\nabla u\|_{L^{\infty}(\Omega_{2r})}+s\big)\\
        &\leq C\big(\frac{\rho}{r}\big)^{\alpha_2}r^{-n}\|\nabla u\|_{L^1(\Omega_r)}+C\overline{h}_1(2\rho)+C\tilde{\varrho}(\rho)\big(\|\nabla u\|_{L^{\infty}(\Omega_{2r})}+s\big),
    \end{align*}
    here the second inequality used this fact that
    \begin{equation*}
        \phi(r_1)\leq Cr^{-n}\|\nabla u\|_{L^1(\Omega_r)},
    \end{equation*}
    and the last inequality used that
    \begin{equation*}
        \hat{h}_1(x,t)\leq \overline{h}_1(x,t),\quad \hat{\varrho}_1(\rho)\leq \tilde{\varrho}_1(\rho).
    \end{equation*}
    Thus \eqref{in+bou phi} yields.

    {\bf Case (ii):} $4\rho<\text{dist}(0,\partial\Omega)<r/4$. Set $r_2=\text{dist}(0,\partial\Omega)/4>\rho$. By using \eqref{in phi rho}, we have
    \begin{align*}
        \phi(\rho)\leq C\big(\frac{\rho}{r_2}\big)^{\alpha_2}\phi(r_2)+C\hat{h}_1(2\rho)+C\hat{\varrho}(\rho)\big(\|\nabla u\|_{L^{\infty}(\Omega_r(y_0))}+s\big).
    \end{align*}
    Now we choose $z_0\in\partial\Omega$ such that $\text{dist}(0,\partial\Omega)=|z_0|$. Then $B_{r_2}\subset B_{5r_2}(z_0)$, $B_r(z_0)\subset B_{2r}$ and by \eqref{bou overline phi rho}, we have
    \begin{align*}
        \phi(r_2)\leq C\overline{\phi}(z_0,5r_2)\leq C\big(\frac{r_2}{r}\big)^{\alpha_2}\overline{\phi}(z_0,r/2)+C\hat{h}_1(z_0,10r_2)
        +C\hat{\varrho}_1(5r_1)\big(\|\nabla u\|_{L^{\infty}(\Omega_r(z_0))}+s\big).
    \end{align*}
    Since $\Omega_r(z_0)\subset\Omega_{2r}$, $\Omega_{r/2}(z_0)\subset\Omega_r$, and $\Omega_{10r_2}(z_0)\subset\Omega_{14r_2}$, combining the two inequalities above and using \cite[Lemma 5.8]{dz2024} yields \eqref{in+bou phi}.

    {\bf Case (iii):} $\text{dist}(0,\partial\Omega)\leq 4\rho$. We choose $z_0\in\partial\Omega$ such that $\text{dist}(0,\partial\Omega)=|z_0|$, then $B_{10\rho}(z_0)\subset B_{14\rho}\subset B_r$. By using \eqref{bou overline phi rho}, we get
    \begin{align*}
        \phi(\rho)\leq C\overline{\phi}(z_0,5\rho)\leq C\big(\frac{\rho}{r}\big)^{\alpha_2}\overline{\phi}(z_0,r/2)+C\hat{h}_1(z_0,10\rho)
         +C\hat{\varrho}_1(5\rho)\big(\|\nabla u\|_{L^{\infty}(\Omega_r(z_0))}+s\big).
    \end{align*}
    Since $\Omega_r(z_0)\subset\Omega_{2r}$, $\Omega_{r/2}(z_0)\subset\Omega_r$ and $\Omega_{10\rho}(z_0)\subset\Omega_{14\rho}$, using \cite[Lemma 5.8]{dz2024}, we get \eqref{in+bou phi}.

    Finally, by replacing $\rho$ with $\varepsilon^j\rho$ in \eqref{in+bou phi} and summing over $j$, using \cite[Lemma 5.8]{dz2024} and the principle of comparison Riemann integrals, $\eqref{in+bou sum phi}$ is proved. The proof of Lemma \ref{Lem in+bou sum phi/psi}
    is finished.
\end{proof}

Now we use Lemma \ref{Lem in+bou sum phi/psi} to prove the global gradient modulus of continuity estimates. We still fix $\varepsilon\in(0,1/4)$ sufficiently small such that
\begin{equation*}
    C\varepsilon^{\alpha-\alpha_2}<1\quad and \quad \varepsilon^{\alpha_2}<1/4,
\end{equation*}
where $C$ is the constant in Proposition \ref{prop-phi-Du} and Corollary \ref{cor bou}, $\alpha\in(0,1)$ is the same as in Theorem \ref{thm v-BMO}, $\alpha_1\in(0,\alpha)$, and $\alpha_2=(\alpha+\alpha_1)/2$. Taking $R\in(0,R')$, as defined in \eqref{def-tildeW} and \eqref{def tilde I}, we still define ${\bf \tilde W}_{1/p,p}$ and ${\bf \tilde I}_1$ as
\begin{align*}
{\bf{\tilde W}}_{\frac{1}{p},p}^\rho(|\mu|)(x)=\sum_{i=1}^{\infty}\varepsilon^{\alpha_1 i}\big({\bf W}_{1/p,p}^{\varepsilon^{-i}\rho}(|\mu|)(x)[\varepsilon^{-i}\rho\leq R/2]+{\bf W}_{1/p,p}^{R/2}(|\mu|)(x)[\varepsilon^{-i}\rho>R/2]\big)
\end{align*}
and
\begin{align*}
    {\bf {\tilde I}}_1^{\rho}(|\mu|)(x)=\sum_{i=1}^{\infty}\varepsilon^{\alpha_1i}\big({\bf I}_1^{\varepsilon^{-i}\rho}(|\mu|)(x)[\varepsilon ^{-i}\rho\leq R/2]+{\bf I}_1^{R/2}(|\mu|)(x)[\varepsilon ^{-i}\rho>R/2]\big),
\end{align*}
here ${\bf W}_{1/p,p}$ and ${\bf I}_1$ are the Wolff and Riesz potentials defined in \eqref{W potential} and \eqref{I potential}, respectively.
\begin{theorem}\label{thm contonuity in+bou}
    Suppose that the conditions of Theorem \ref{bou pointwise estiamte} are satisfied, $\alpha\in(0,1)$ is the constant in Theorem \ref{thm v-BMO}, $\alpha_1\in(0,\alpha)$. Then there exist constants $R'=R'(R_0,\varrho_0)\in(0,R_0)$ and $C=C(n, p, \lambda, \alpha_1, \varrho, R_0, \varrho_1)$, such that for any $x_0\in\overline{\Omega}$, $R\in(0,R']$, and any Lebesgue points $x$, $y$ of the vector-valued function $\nabla u$, the following estimates hold:

    (i) when $p\geq2$, then
    \begin{align}\label{continuty p>2}
        |\nabla u(x)-\nabla u(y)|\leq C\mathcal{M}_4\Bigg(\big(\frac{\rho}{R}\big)^{\alpha_1}+\int_0^{\rho}\frac{\tilde{\varrho}_1(t)}{t} \ dt\Bigg)+C\|{\bf {\tilde W}}_{1/p,p}^{\rho}(|\mu|)\|_{L^{\infty}(\Omega_{R/4}(x_0))},
    \end{align}
    where $\rho=|x-y|$, $\varrho_1=\varrho+\varrho_0$, and 
    \begin{equation*}
        \mathcal{M}_4:=\|W_{1/p,p}^R(|\mu|)\|_{L^{\infty}(\Omega_R(x_0))}+C\fint_{\Omega_R(x_0)}\big(|\nabla u|+s\big)\ dy;
    \end{equation*}
    
    (ii) when $3/2\leq p<2$, then
    \begin{align*}
        |\nabla u(x)-\nabla u(y)|&\leq C\mathcal{M}_5\Bigg(\big(\frac{\rho}{R}\big)^{\alpha_1}+\int_0^{\rho}\frac{\overline{\varrho}_1(t)}{t}\ dt\Bigg)+C\|{\bf \tilde W}_{1/p,p}^{\rho}(|\mu|)\|_{L^{\infty}(\Omega_{R/4}(x_0))}\nonumber\\
        &\quad +C\mathcal{M}_5^{2-p}\|{\bf \tilde I}_1^{\rho}(|\mu|)\|_{L^{\infty}(\Omega_{R/4}(x_0))},
    \end{align*}
    where $\rho=|x-y|$, $\varrho_1=\varrho+\varrho_0$, and
    \begin{equation*}
        \mathcal{M}_5:=\|I_1^R(|\mu|)\|_{L^{\infty}(\Omega_{R/4}(x_0))}^{\frac{1}{p-1}}+\Bigg(\fint_{\Omega_R(x_0)}\big(|\nabla u|+s\big)^{2-p}\ dy\Bigg)^{\frac{1}{2-p}};
    \end{equation*}
     (iii) when $1<p<3/2$, then
    \begin{align*}
        |\nabla u(x)-\nabla u(y)|&\leq C\mathcal{M}_6\Bigg(\big(\frac{\rho}{R}\big)^{\alpha_1}+\int_0^{\rho}\frac{\overline{\varrho}_1(t)}{t}\ dt\Bigg)+C\|{\bf \tilde W}_{1/p,p}^{\rho}(|\mu|)\|_{L^{\infty}(\Omega_{R/4}(x_0))}\nonumber\\
        &\quad +C\mathcal{M}_6^{2-p}\|{\bf \tilde I}_1^{\rho}(|\mu|)\|_{L^{\infty}(\Omega_{R/4}(x_0))},
    \end{align*}
    where $\rho=|x-y|$, $\varrho_1=\varrho+\varrho_0$, and
    \begin{equation*}
        \mathcal{M}_6:=\|I_1^R(|\mu|)\|_{L^{\infty}(\Omega_{R/4}(x_0))}^{\frac{1}{p-1}}+\Bigg(\fint_{\Omega_R(x_0)}\big(|\nabla u|+s\big)^{\frac{(p-1)^2}{2}}\ dy\Bigg)^{\frac{2}{(p-1)^2}}.
    \end{equation*}
\end{theorem}
    \begin{proof} We will prove the case of $p\geq 2$ as an example since other cases are similar.
For any Lebesgue points $x$, $y$ $\in\Omega_{R/4}(x_0)$ of $\nabla u$, we denote $\rho:=|x-y|>0$.  If $\rho\geq R/16$, then $\rho/R\geq1/16$, which  yields
    \begin{align*}
        |\nabla u(x)-\nabla u(y)|\leq 2\|\nabla u\|_{L^{\infty}(\Omega_{R/2}(x_0))}\leq \frac{32\rho}{R}\|\nabla u\|_{L^{\infty}(\Omega_{R/2}(x_0))}\leq C\big(\frac{\rho}{R}\big)^{\alpha_2}\|\nabla u\|_{L^{\infty}(\Omega_{R/2}(x_0))}.
    \end{align*} 
    
   If $\rho<R/16$, using the triangle inequality, we get
    \begin{align}\label{triangle nabla ux-nabla uy}
        |\nabla u(x)-\nabla u(y)|&\leq |\nabla u(x)-{\bf q}_{x,\rho}|+|{\bf q}_{x,\rho}-{\bf q}_{y,2\rho}|+|{\bf q}_{y,2\rho}-\nabla u(y)|\nonumber\\
        &\leq |\nabla u(x)-{\bf q}_{x,\rho}|+|\nabla u(y)-{\bf q}_{x,2\rho}|+|\nabla u(z)-{\bf q}_{x,\rho}|+|\nabla u(z)-{\bf q}_{y,2\rho}|.
    \end{align}
   Similar to the proof \eqref{nabla-q w}, we have
    \begin{equation*}
        |\nabla u(x)-{\bf q}_{x,\rho}|\leq C\sum_{j=0}^{\infty}\phi(x,\varepsilon^j\rho)\quad and\quad |\nabla u(y)-{\bf q}_{x,2\rho}|\leq C\sum_{j=0}^{\infty}\phi(y,2\varepsilon^j\rho),
    \end{equation*}
    where $C$ is a constant that depends only on $n$.
    Since $\Omega_{\rho}(x)\subset\Omega_{2\rho}(y)$, averaging the integral over $z\in\Omega_{\rho}(x)$ in \eqref{triangle nabla ux-nabla uy}, we get
    \begin{align*}
        |\nabla u(x)-\nabla u(y)|&\leq |\nabla u(x)-{\bf q}_{x,\rho}|+|\nabla u(y)-{\bf q}_{x,2\rho}|+C\phi(x,\rho)+C\phi(y,2\rho)\\
        &\leq C\sum_{j=0}^{\infty}\phi(x,\varepsilon^j\rho)+C\sum_{j=0}^{\infty}\phi(y,2\varepsilon^j\rho)+C\phi(x,\rho)+C\phi(y,2\rho)\\
        &\leq C\sup_{z_0\in\Omega_{R/4}(x_0)}\sum_{j=0}^{\infty}\phi(z_0,2\varepsilon^j\rho).
    \end{align*}
Since $\Omega_{R/4}(z_0)\subset\Omega_{R/2}(x_0)$, $\forall z_0\in\Omega_{R/4}(x_0)$, replacing $r$ by $R/8$ in \eqref{in+bou sum phi}, we obtain
    \begin{align}\label{|nabla ux-nabla uy| rho<R/16}
        |\nabla u(x)-\nabla u(y)|&\leq C\big(\frac{\rho}{R}\big)^{\alpha_2}\|\nabla u\|_{L^{\infty}(\Omega_{R/2}(x_0))}+C\sup_{z_0\in\Omega_{R/4}(x_0)}\int_0^{\rho}\frac{\overline{h}_1(z_0,t)}{t}\ dt\nonumber\\
        &\quad +C\big(\|\nabla u\|_{L^{\infty}(\Omega_{R/2}(x_0))}+s\big)\int_0^{\rho}\frac{\tilde{\varrho}_1(t)}{t}\ dt.
    \end{align}
 Note that
    \begin{equation*}
        \overline{h}_1(y_0,t)\leq \sum_{i=1}^{\infty}\varepsilon^{\alpha_1i}\big(h(x,\varepsilon^{-i}t)[\varepsilon^{-i}t\leq R/2]+h(x_0,R/2)[\varepsilon^{-i}t>R/2]\big).
    \end{equation*}
   Then  for any $z_0\in\Omega_{R/4}(x_0)$ and $\rho\in(0,R/2)$, similar to the proof of \eqref{sum tilde h}, we have
    \begin{equation*}
        \int_0^{\rho}\frac{\overline{h}_1(z_0,t)}{t}\ dt\leq {\bf {\tilde W}}_{1/p,p}^{\rho}(|\mu|)(z_0)+C\big(\frac{\rho}{R}\big)^{\alpha_1}{\bf W}_{1/p,p}^R(|\mu|)(z_0).
    \end{equation*}
    From \eqref{bou p>2 point} we get
    \begin{equation*}
        \|\nabla u\|_{L^{\infty}(\Omega_{R/2}(x_0))}\leq C\|{\bf W}_{1/p,p}^R(|\mu|)(x)\|_{L^{\infty}(\Omega_R(x_0))}+C\fint_{\Omega_R(x_0)}\big(|\nabla u|+s\big)\ dy.
    \end{equation*}
     Combining \eqref{|nabla ux-nabla uy| rho<R/16} and the last two inequalities, we arrive at \eqref{continuty p>2}. Theorem \ref{thm contonuity in+bou} is proved.
    \end{proof}

\section*{Acknowledgement}

The first author would like to thank Prof. Hongjie Dong for very helpful discussions and suggestions.

\bibliographystyle{abbrv}
\bibliography{ref}

\end{document}